\documentclass{amsart}
\usepackage[OT2,T1]{fontenc}
\usepackage[russian,english]{babel}
\usepackage{color}
\usepackage{epsfig}
\usepackage{amsmath}
\usepackage{amssymb}
\usepackage{amsthm}
\usepackage{amsfonts}
\usepackage{euscript}
\usepackage[matrix,arrow]{xy}
\usepackage{mathrsfs}
\usepackage{setspace}

\usepackage{dingbat}

\numberwithin{equation}{section}
\def\Tiny{\fontsize{4pt}{4pt}\selectfont}

\newcommand{\bC}{{\mathbb C}}

\newcommand{\bF}{{\mathbb F}}

\newcommand{\bK}{{\mathbb K}}

\newcommand{\bP}{{\mathbb P}}
\newcommand{\bQ}{{\mathbb Q}}
\newcommand{\bR}{{\mathbb R}}

\newcommand{\bZ}{{\mathbb Z}}

\newcommand{\scrA}{\EuScript A}
\newcommand{\scrB}{\EuScript B}
\newcommand{\scrC}{\EuScript C}
\newcommand{\scrD}{\EuScript D}

\newcommand{\scrM}{\EuScript M}
\newcommand{\scrN}{\EuScript N}

\newcommand{\scrW}{\EuScript W}

\renewcommand{\hom}{\operatorname{\mathit{hom}}}

\renewcommand{\mod}{\operatorname{\mathit{mod}}}
\newcommand{\modt}{\operatorname{\mathit{modt}}}
\newcommand{\Fuk}{\mathfrak{F}}
\newcommand{\Ch}{\operatorname{\mathit{Ch}}}
\newcommand{\HF}{\operatorname{\mathit{HF}}}
\newcommand{\HW}{\operatorname{\mathit{HW}}}
\newcommand{\SH}{\operatorname{\mathit{SH}}}
\newcommand{\tw}{\operatorname{\mathit{tw}}}

\newcommand{\Diff}{\mathrm{Diff}}

\newcommand{\iso}{\cong}
\newcommand{\htp}{\simeq}

\newcommand{\complexity}{\text{\selectlanguage{russian}\CYREREV}}
\newcommand{\scrMan}{\scrM^{\raisebox{-.2ex}[0ex][0ex]{\textrm{{\Tiny\anchor}}}}}

\newtheorem{thm}{Theorem}[section]
\newtheorem{theorem}[thm]{Theorem}
\newtheorem{property}[thm]{Property}

\newtheorem{corollary}[thm]{Corollary}

\newtheorem{assumption}[thm]{Assumption}

\newtheorem{remark}[thm]{Remark}

\newtheorem{example}[thm]{Example}

\newtheorem{lemma}[thm]{Lemma}

\begin{document}
\title[Homologous recombination]{Altering symplectic manifolds\\ by homologous recombination}
\author[Abouzaid, Seidel]{Mohammed Abouzaid, Paul Seidel}
\date{August 2, 2010}
\maketitle

\section{Introduction}

\subsection{A quick review}
This paper studies the non-uniqueness of symplectic structures on open manifolds of high dimension. More precisely, we only consider symplectic manifolds which are Liouville, and for the most part, ones which are of finite type (a brief summary of the definitions is given in Section \ref{subsec:def}). Non-specialist readers may want to keep in mind that every smooth complex affine variety can be turned into a finite type Liouville manifold by choosing a suitable K{\"a}hler form. More generally, every Stein manifold can be made Liouville, but the outcome may not always be of finite type. It almost goes without saying that the basic relation between Stein and symplectic geometry, as well as many of the fundamental ideas in this field, were pioneered by Eliashberg and Gromov \cite{eliashberg90, eliashberg-gromov93, eliashberg94}.

Let us turn to the specific non-uniqueness question, beginning with the case of flat space (we omit the earlier history of exotic symplectic structures on flat space, since those are not known to be Liouville). Seidel-Smith \cite{seidel-smith04b} exhibited a finite type Liouville structure on $\bR^{2n}$, for any even $n \geq 4$, which is not symplectomorphic to the standard one. Examples for any $n \geq 4$ were constructed by McLean \cite{mclean09}. In fact, the last-mentioned paper proves the much stronger statement that for any $n \geq 4$, there is an infinite sequence of finite type Liouville structures on $\bR^{2n}$ which are pairwise non-symplectomorphic.

Among other open manifolds, cotangent bundles have received the most attention. The best available result is again due to McLean \cite{mclean09}. It says that on the cotangent bundle of any closed manifold of dimension $n \geq 4$, there is an infinite sequence of finite type Liouville structures which are pairwise non-symplectomorphic. More recently, Maydanskiy-Seidel \cite{maydanskiy-seidel09} (based on earlier work of Maydanskiy \cite{maydanskiy09}) constructed another nonstandard Liouville structure on the cotangent bundle of the $n$-sphere for any $n \geq 3$, which is different from McLean's because it does not contain any Lagrangian sphere. Yet another construction for $T^*S^n$, which is at least philosophically close to the last-mentioned one but maybe more explicit, is given in Bourgeois-Ekholm-Eliashberg \cite{bourgeois-ekholm-eliashberg09}. In all these cases, the nonstandard structures are almost symplectomorphic to the standard ones, which means that they are undistinguishable in ``soft'' or homotopy-theoretic terms (this is a traditional concept, recalled in Section \ref{subsec:almost}).

\subsection{Finite type Liouville manifolds\label{subsec:results}}
Following a well-established direction, the main invariant we will use is symplectic cohomology, in the sense of \cite{viterbo97a} (an earlier related construction is \cite{cieliebak-floer-hofer95}). This associates to each Liouville manifold a $\bZ/2$-graded $\bK$-vector space, where $\bK$ is some arbitrary coefficient field (the precise sense in which this is an invariant is discussed in Section \ref{subsec:def-sh}).

\begin{theorem} \label{th:1}
Let $X$ be a smooth complex affine variety of real dimension $\geq 6$. Then there is a finite type Liouville manifold $\tilde{X}$ which is almost symplectomorphic to $X$, and such that the symplectic cohomology of $\tilde{X}$ vanishes (for all $\bK$).
\end{theorem}

This is partly based on the argument used in \cite{maydanskiy-seidel09} for $T^*S^n$, which in algebro-geometric terms would be the special case of an affine quadric. The algebraic nature of $X$ is important only insofar as it allows us to apply basic tools from symplectic Picard-Lefschetz theory. Concretely, we represent our manifold as the total space of a Lefschetz fibration, and then manipulate that description to construct the new space $\tilde{X}$ (this is the homologous recombination process mentioned in the title; the name comes from biology, see for instance \cite{capecchi89}). A different proof of Theorem \ref{th:1} has been obtained independently by Bourgeois-Ekholm-Eliashberg (currently unpublished, but based on the stabilization process from \cite[Section 6.2]{bourgeois-ekholm-eliashberg09}). In fact, their approach avoids Lefschetz fibrations, hence has the substantial advantage of applying to all finite type Weinstein manifolds $X$.

By combining Theorem \ref{th:1} with the argument from \cite{mclean09}, one arrives at the following conclusion (also independently noticed by Smith):

\begin{corollary} \label{th:1.5}
Take $X$ as in Theorem \ref{th:1}. Then there is an infinite sequence of finite type Liouville manifolds $\tilde{X}_k$, all of which are almost symplectomorphic to $X$, and such that $\tilde{X}_k$ is not symplectomorphic to $\tilde{X}_l$ for $k \neq l$.
\end{corollary}

Even though we have not mentioned this so far, the freedom to use different coefficient fields $\bK$ is important since it allows for a more delicate version of homologous recombination:

\begin{theorem} \label{th:2}
Let $X$ be a smooth complex affine variety of real dimension $\geq 12$. Fix an integer $q \geq 1$. Then there is a finite type Liouville manifold $\tilde{X}$ almost symplectomorphic to $X$, and with the following property. If $\bK$ is a field whose characteristic divides $q$, the symplectic cohomology of $\tilde{X}$ with coefficients in $\bK$ vanishes. On the other hand, if the characteristic of $\bK$ does not divide $q$ (this includes characteristic $0$), the symplectic cohomology of $\tilde{X}$ with coefficients in $\bK$ vanishes if and only if the same holds for the original manifold $X$.
\end{theorem}

\begin{remark}
Symplectic cohomology can actually be defined with coefficients in $\bZ$ (or any other abelian group) as well. In the situation of Theorem \ref{th:2}, let us temporarily denote the integral version by $\SH^*(\tilde{X};\bZ)$, and its analogue with coefficients in $\bF_p$ by $\SH^*(\tilde{X};\bF_p)$. When considered as $\bZ/2$-graded abelian groups, these fit into a long exact sequence
\begin{equation}
\cdots \longrightarrow \SH^*(\tilde{X};\bZ) \xrightarrow{p} \SH^*(\tilde{X};\bZ) \longrightarrow \SH^*(\tilde{X};\bF_p) \longrightarrow \cdots
\end{equation}
Hence, vanishing of $\SH^*(\tilde{X};\bF_p)$ is equivalent to unique $p$-divisibility for $\SH^*(\tilde{X};\bZ)$.
\end{remark}

All of the symplectic structures constructed so far are actually Weinstein, which means that they admit Morse functions compatible with the Liouville structure. Let us introduce a notion of complexity, as the minimal number of critical points of such a function (see Section \ref{subsec:def-wein}). With this in mind, we can give a sharper version of Corollary \ref{th:1.5} for manifolds of sufficiently high dimension:

\begin{corollary} \label{th:2.5}
Take $X$ as in Theorem \ref{th:2}. Then there is an infinite sequence of manifolds as in Corollary \ref{th:1.5}, but with the additional property that all of them are of bounded complexity.
\end{corollary}

This contrasts with the approach from \cite{mclean09} in which one takes repeated boundary connect sums that, at least intuitively, would appear to increase complexity (one has to be careful here since no lower bound for the complexity of such a sum has actually been proved; establishing such bounds is a major open problem in general).


\subsection{Infinite type Liouville manifolds}
McLean \cite{mclean06} constructed an example of a Liouville structure on $\bR^{2n}$, $n \geq 4$, which is not of finite type (because its symplectic cohomology is of uncountable dimension as a vector space, whereas that of any finite type Liouville manifold is at most of countable dimension). By combining his idea with homologous recombination, we see:

\begin{theorem} \label{th:3}
Fix an arbitrary set of primes $\bP$. For any $n \geq 6$, there is a Liouville manifold diffeomorphic to $\bR^{2n}$, whose symplectic cohomology with coefficients in $\bK = \bF_p$ is of countable dimension if and only if $p \in \bP$.
\end{theorem}

\begin{corollary} \label{th:3.5}
Let $M$ be a finite type Liouville manifold of dimension $\geq 12$, with the property that the map
$H^1_{\mathrm{cpt}}(M;\bR) \rightarrow H^1(M;\bR)$ is onto. Then there is an uncountable family of Liouville manifolds which are almost symplectomorphic to $M$, but pairwise non-symplectomorphic.
\end{corollary}

The assumption on the first cohomology is somewhat peripheral to the whole argument, and can be removed if one is willing to replace symplectomorphism by the stricter notion of exact symplectomorphism (see Section \ref{subsec:def}). Theorem \ref{th:3} and Corollary \ref{th:3.5} were proved independently by McLean \cite{mclean08}, whose approach is somewhat different since it relies heavily on the ring structure of symplectic cohomology, following \cite{mclean09}. In fact, McLean's argument should also cover the dimensions $10$ and $8$, and possibly $6$ as well, which are missing from our results.

\subsection{Acknowledgments}
This research was conducted while the first author was a Clay Research Fellow. The second author was partially supported by NSF grant DMS-1005288. We thank Kai Cieliebak for letting us have the manuscript of his book in progress with Eliashberg, and Mark McLean for explaining some of his unpublished work to us.

\section{Concepts and tools}

\subsection{Exact symplectic manifolds\label{subsec:def}}
We start by introducing the main relevant classes of symplectic manifolds. Some other references covering similar basic material are \cite{eliashberg-gromov93, eliashberg94, seidel-smith04b, bourgeois-ekholm-eliashberg09} (there is also a forthcoming book \cite{cieliebak-eliashberg}, which looks likely to become the definitive reference on the subject). All our symplectic manifolds $M$ will be exact, which means that they come with a distinguished primitive $\theta_M$ for the symplectic form $\omega_M = d\theta_M$. The Liouville vector field $Z_M$ is defined by $\omega_M(Z_M,\cdot) = \theta_M$. We say that the exact symplectic structure is complete if the flow of $Z_M$ is defined for all time.

\begin{itemize} \parskip1em
\item A {\em Liouville domain} is a compact exact symplectic manifold with boundary, such that $Z_M$ points strictly outwards along $\partial M$.

\item A {\em Liouville manifold} is a complete exact symplectic manifold, which admits an exhausting (proper and bounded below) function $h: M \rightarrow \bR$ such that the following holds. There is a sequence of numbers $c_k \in \bR$, going to $+\infty$ as $k \rightarrow \infty$, such that $dh(Z_M) > 0$ along $h^{-1}(c_k)$ (note that the corresponding sublevel sets yield an exhaustion of $M$ by Liouville domains).

\item A Liouville manifold is called of {\em finite type} if in fact, $dh(Z_M) > 0$ outside a compact subset.
\end{itemize}

Any Liouville domain can be canonically enlarged to a finite type Liouville manifold by attaching an infinite cone to the boundary. Conversely, any finite type Liouville manifold is obtained in this way (one truncates it to a Liouville domain which is a sufficiently large sublevel set of the exhausting function $h$).

There are several variations on the notion of symplectic isomorphism between two such manifolds.

\begin{itemize} \parskip1em
\item Two Liouville domains $M$ and $\tilde{M}$ are called {\em deformation equivalent} if there is a diffeomorphism $\phi: M \rightarrow \tilde{M}$ and a one-parameter family of Liouville domain structures on $M$, which interpolate between $(\omega_M,\theta_M)$ and $(\phi^*\omega_{\tilde{M}},\phi^*\theta_{\tilde{M}})$.

\item An {\em exact symplectomorphism} between Liouville manifolds is a diffeomorphism $\phi: M \rightarrow \tilde{M}$ such that $\phi^*\theta_{\tilde{M}} - \theta_M$ is an exact one-form.

\item Now, suppose that $M$ and $\tilde{M}$ are Liouville manifolds of finite type. A {\em strictly exact symplectomorphism} $\phi: M \rightarrow \tilde{M}$ is a diffeomorphism such that $\phi^*\theta_{\tilde{M}} - \theta_M$ is the derivative of a compactly supported function.
\end{itemize}

(We have omitted some intermediate possibilities.) If we have two deformation equivalent Liouville domains, their enlargements are Liouville manifolds which are symplectomorphic in the strictly exact sense, and the converse is also true. Next, take two Liouville manifolds, at least one of which is of finite type.  Then, every symplectic isomorphism between them can be deformed to an exact one, by combining it with a suitable non-Hamiltonian flow \cite[Lemma 1.1]{bourgeois-ekholm-eliashberg09}. Inspection of the argument shows that one can replace the finite type condition by the surjectivity of the map $H^1_{\mathrm{cpt}}(M;\bR) \rightarrow H^1(M;\bR)$.

\subsection{Weinstein manifolds\label{subsec:def-wein}}
There is a special class of Liouville manifolds, which is somewhat closer to the original motivation from Stein geometry \cite{eliashberg90}.

\begin{itemize}\parskip1em
\item A {\em Weinstein manifold} is a complete exact symplectic manifold, which admits an exhausting Morse function $h$ such that the following holds. Outside the critical point set of $h$, $dh(Z_M) > 0$. At each critical point $x$ of $h$, $Z_M = 0$, and the quadratic form on $TM_x$ given by $X \mapsto D^2 h(X,DZ_M(X))$ is positive definite (this is a weak version of the notion of gradient-like vector field; it is equivalent to the formulation in \cite{bourgeois-ekholm-eliashberg09}).

\item A Weinstein manifold is called {\em of finite type} if $\theta_M$ vanishes only at finitely many points (by definition, these will be the critical points of any Morse function $h$ with the properties described above).

\item A {\em Weinstein domain} is a compact exact symplectic manifold with boundary, such that the following holds. It admits a Morse function $h: M \rightarrow (-\infty,0]$ such that $h^{-1}(0) = \partial M$ is a regular level set, and which otherwise has the same properties as for a Weinstein manifold.
\end{itemize}

We need to elaborate a little bit more on the relation between the two notions. Consider a Liouville domain $M$, and its standard enlargement, obtained by attaching an infinite cone $[1,\infty) \times \partial M$, on which the one-form is $r(\theta|\partial M)$, with $r$ the radial variable. In one direction, suppose that $M$ is a Weinstein domain, with a function $h$ as in the definition. One can extend this function to the cone, by sending $(r,x) \mapsto h(x) + (r-1)(Z_M.h)(x)$ for $(r,x) \in [1,\infty) \times \partial M$. The extension is exhausting, and its derivative in the direction of the Liouville vector field is positive on the cone. It is only $C^1$ along $\{1\} \times \partial M$, but one can easily smooth it, and this shows that the enlargement is a Weinstein manifold of finite type. In the other direction, we start with a general Liouville domain $M$, and suppose that its enlargement is a finite type Weinstein manifold, with corresponding exhausting Morse function $h$. Since $\partial_r h > 0$ outside a compact subset, a sufficiently large level set $h^{-1}(c)$ will be the graph of a function $\partial M \rightarrow [1,\infty)$. Hence, while it is not clear whether $M$ itself will be a Weinstein domain, it is certainly deformation equivalent to one within the class of Liouville domains.

Given a Liouville domain $M^0$ and an embedding of a sphere $S^{r-1} \hookrightarrow \partial M^0$ such that the pullback of $\theta_{M^0}$ to the sphere vanishes,  one can define a larger Liouville domain $M^1 \supset M^0$, depending on a choice of framing data, by attaching a Weinstein $r$-handle to the boundary \cite{weinstein91}. As an example, suppose that $M^0$ has two connected components. One can then take one point on the boundary of each component, which together form a sphere of the trivial dimension $r-1 = 0$, and attach a Weinstein one-handle. This is called taking the boundary connect sum of the two components. The other important case for us is the largest possible value $r = \mathrm{dim}(M)/2$. These are called critical Weinstein handles, and will be reviewed in more detail in Section \ref{subsec:weinstein} below.
If $M^0$ is a Weinstein domain, then the corresponding Morse function can be extended over $M^1$ (acquiring one more critical point), which is then again a Weinstein domain. Conversely, each Weinstein domain can be written as the result of starting with a disjoint union of balls and attaching finitely many Weinstein handles.

We will occasionally use a more quantitative version of the remarks above. Let $M$ be a finite type Liouville manifold. Suppose that there is a compactly supported function $k$ such that the modified one-form $\theta_M + dk$ is Weinstein of finite type. We define the {\em complexity} $\complexity(M)$ to be the minimal number of zeros of $\theta_M + dk$, taken over all such $k$ (if there is none, the complexity is set to infinity). By definition, complexity is invariant under strictly exact symplectomorphisms, and therefore can also be considered as an invariant of Liouville domains, up to deformation equivalence.
In the case where $M^1$ is obtained from $M^0$ by attaching a single Weinstein handle, we have
\begin{equation} \label{eq:complexity-one}
\complexity(M^1) \leq \complexity(M^0) + 1.
\end{equation}

\begin{example} \label{th:affine}
There is a fundamental class of examples \cite{eliashberg-gromov93}, which we recall in order to have its properties handy for reference. Let $X \subset \bC^N$ be a smooth affine variety. For generic $x_0 \in \bC^N$, the function $h(x) = \frac{1}{4} \|x - x_0\|^2$ is Morse \cite[\S 6]{milnor63}. Moreover, since it is real algebraic, it can have only finitely many critical points. If we take $\theta_X = -dh \circ I_X$ ($I_X$ being the complex structure), then $\omega_X = d\theta_X $ is the restriction of the standard constant symplectic form on $\bC^{N}$ to $X$. The Liouville vector field $Z_X$ is just the gradient vector field of $h$. Since that satisfies $\|\nabla h\|^2 \leq h$, it is also complete. Hence, equipping $X$ with these structures makes it into a Weinstein manifold of finite type.

It is worth while emphasizing that many other possible choices lead to equivalent results. Let $\tilde{h}$ be any exhausting plurisubharmonic function on $X$. This means that if we take $\tilde{\theta}_X = -d\tilde{h} \circ I_X$, then $\tilde{\omega}_X = d\tilde{\theta}_X$ is K{\"a}hler. Suppose in addition that this exact symplectic structure is complete (which, by the way, can always be achieved by reparametrizing a given exhausting plurisubharmonic function \cite[Lemma 3.1]{biran-cieliebak01}). By a result from \cite[Part IV]{cieliebak-eliashberg}, any two such structures are exact symplectomorphic. In particular, $(\tilde{\omega}_X,\tilde{\theta}_X)$ is exact symplectomorphic to the previously defined $(\omega_X,\theta_X)$.
\end{example}


\subsection{Symplectic cohomology\label{subsec:def-sh}}
We will summarize some important properties of this invariant in a black box fashion. Besides the original papers \cite{viterbo97a, viterbo97b}, a few references are \cite{cieliebak02, oancea04b, seidel07, bourgeois-oancea09, cieliebak-frauenfelder-oancea09, ritter10}. For any Liouville domain $M$ and coefficient field $\bK$, symplectic cohomology yields a $\bZ/2$-graded $\bK$-vector space $\SH^*(M)$, which is of at most countable dimension over $\bK$. Moreover, if $\epsilon: U \rightarrow M$ is an embedding of one such domain into another of the same dimension, such that $\epsilon^*\theta_{M} - c\theta_{U}$ is exact for some constant $c>0$ (which implies that the embedding is conformally symplectic), we get a restriction map $\SH^*(M) \rightarrow \SH^*(U)$. This is functorial with respect to compositions of embeddings. Moreover, it is invariant under isotopies of embeddings, within the same class (this is proved by a parametrized version of the argument which constructs the restriction maps). As an elementary consequence of these properties, deformation equivalent Liouville domains have isomorphic symplectic cohomology groups. Via the correspondence between Liouville domains and finite type Liouville manifolds, one then defines symplectic cohomology for the latter class. The outcome is clearly invariant under strictly exact symplectic isomorphisms; by thinking a little about conformally symplectic embeddings, one sees that it is in fact invariant under exact symplectic isomorphisms; and hence, by the observation made previously, under general symplectic isomorphisms of finite type Liouville manifolds.

More generally, for any Liouville manifold $M$ which is not necessarily of finite type, one defines
\begin{equation} \label{eq:inverse-limit}
\SH^*(M) = \underleftarrow{\mathit{lim}} \SH^*(U)
\end{equation}
where the limit is over all pairs $(U,k)$ consisting of a subdomain (a compact codimension zero submanifold) $U \subset M$ and a function $k$ on $M$, such that the Liouville vector field of $\theta_M + dk$ points strictly outwards along $\partial U$. Note that by definition, there is a cofinal family for which $k = 0$, so one could also restrict attention to such subdomains. Moreover, in the finite type case there is a cofinal family for which the inverse system is constant, and we therefore recover the previous definition. One sees from \eqref{eq:inverse-limit} that $\SH^*(M)$ is invariant under exact symplectic isomorphisms. It is not clear whether, in the infinite type case, it is actually invariant under general symplectic isomorphisms. However, this is true for manifolds with the property that $H^1_{\mathrm{cpt}}(M;\bR) \rightarrow H^1(M;\bR)$ is surjective.

Symplectic cohomology carries a rich set of additional algebraic structures (see for instance \cite{seidel07, ritter10}). Most importantly for our purpose, it is a unital graded commutative ring. Moreover, there is a canonical map from ordinary to symplectic cohomology, which is compatible with the ring structures.

\subsection{Lefschetz fibrations and Lagrangian Floer cohomology\label{subsec:lefschetz}}
Let $M$ be a Liouville domain of dimension $2n$, and suppose that we are given an ordered collection of Lagrangian spheres $(V_1,\dots,V_m)$. More precisely, each $V_i$ should be Lagrangian, exact (which is trivial unless $n = 1$), and should come with a diffeomorphism $S^n \rightarrow V_i$, unique up to isotopy and composition with elements of $O(n+1)$ (the additional data provided by such a diffeomorphism is trivial for $n \leq 3$ thanks to \cite{smale59, cerf68}). One can then build an exact Lefschetz fibration over the disc with fibre $M$ such that the $V_i$ (in the given order) form a basis of vanishing cycles. After rounding off the corners and a minor manipulation of the symplectic form, the total space of this Lefschetz fibration is a Liouville domain $E$ of dimension $2n+2$. Auxiliary choices are required in order to construct $E$, but the result is unique up to deformation. In addition, it comes with a collection of Lefschetz thimbles $(\Delta_1,\dots,\Delta_m)$, which are Lagrangian discs in $E$ with Legendrian boundaries, determined by the collection $(V_1,\dots,V_m)$.

\begin{remark} \label{th:lefschetz-handles}
Alternatively, this construction can be thought of as a special case of critical Weinstein handle attachment, as follows. Take $D^2 \times M$ and round off the corners, forming a Liouville domain. After a suitable deformation of the exact symplectic structure, each vanishing cycle $V_i$ gives rise to a Legendrian sphere $\Lambda_i$ in the boundary of this domain. Attaching Weinstein handles to those yields a Liouville domain deformation equivalent to $E$ (see \cite[Section 7]{bourgeois-ekholm-eliashberg09} for a more detailed discussion). In view of \eqref{eq:complexity-one}, this implies a bound on the complexity: \rm{
\begin{equation} \label{cor:complexity-fibration}
\complexity(E) \leq \complexity(M) + m.
\end{equation}}
\end{remark}

Fix a coefficient field $\bK$. Even though the computation of symplectic cohomology is our ultimate target, various versions of Lagrangian Floer cohomology will play a crucial role along the way. Choose orientations and (necessarily trivial) {\em Spin} structures for the Lefschetz thimbles, which then induce corresponding structures on the vanishing cycles. We will consider the Floer cohomology groups $\HF^*(V_i,V_j)$ taken inside the fibre $M$, as well as the wrapped Floer cohomology \cite{fukaya-seidel-smith07b, abouzaid-seidel07} groups $\HW^*(\Delta_i,\Delta_j)$ in $E$. Both are $\bZ/2$-graded $\bK$-vector spaces. In addition to standard Floer theory techniques, four relations between these invariants will play a crucial role.

\begin{property} \label{th:open-closed}
$\SH^*(E)$ vanishes if and only if $\HW^*(\Delta_i,\Delta_i)$ vanishes for all $i$.
\end{property}

As one can see from Section \ref{subsec:results}, vanishing or nonvanishing of symplectic cohomology is essential in all our arguments. Property \ref{th:open-closed} reduces this to the corresponding question for wrapped Floer cohomology, which can be tackled one thimble at a time.

\begin{property} \label{th:hull}
Suppose that $\HW^*(\Delta_1,\Delta_1) \neq 0$. Then, for every exact, oriented and {\em Spin} Lagrangian submanifold $W \subset M$ (either closed or with Legendrian boundary) we have
\begin{multline} \label{eq:ineq}
\dim_\bK \, \HF^*(W,V_1)  \leq \\ \sum  \Big(\dim_\bK \HF^*(W,V_{i_r}) \cdot \dim_\bK \HF^*(V_{i_r},V_{i_{r-1}}) \cdots \dim_\bK \HF^*(V_{i_2},V_{i_1}) \Big),
\end{multline}
where the sum is over all $r > 1$ and $1 = i_1 < i_2 < \cdots < i_r \leq m$.
\end{property}

The informal meaning is that the wrapped Floer cohomology of $\Delta_1$ will vanish if $V_1$ is sufficiently independent of the other vanishing cycles. In order to be able to iterate this argument, the following observation is useful. Consider the same fibre $M$, but use only the vanishing cycles $(V_2,\dots,V_m)$. The Lefschetz fibration constructed in this way will be denoted by $\tilde{E}$, and its Lefschetz thimbles by $(\tilde{\Delta}_2,\dots,\tilde{\Delta}_m)$.

\begin{property} \label{th:iterate}
Suppose that $\HW^*(\tilde{\Delta}_i,\tilde{\Delta}_i) = 0$ for all $i$, and $\HW^*(\Delta_1,\Delta_1) = 0$ as well. Then, all $\HW^*(\Delta_i,\Delta_i)$ vanish.
\end{property}

Finally, we have the following meta-principle, which says that the wrapped Floer cohomology of Lefschetz thimbles in $E$ can in principle always be computed in terms of pseudo-holomorphic curves in the fibre $M$:

\begin{property} \label{th:meta}
The wrapped Floer cohomology groups $\HW^*(\Delta_i,\Delta_i)$ depend only on the quasi-isomorphism type of the full $A_\infty$-subcategory of the Fukaya category of $M$ having objects $(V_1,\dots,V_m)$.
\end{property}

The rest of this section concerns the derivations of these four properties. Each of them is a combination of known results (in fact, in some cases there is more than one possible approach).

\subsection{Wrapping and Dehn twists\label{subsec:wrap1}}
We begin with Property \ref{th:hull}. This is a minor variation of \cite[Proposition 6.2]{maydanskiy-seidel09}, hence we will only outline the argument. Let $\Fuk(M)$ be the Fukaya category of the fibre, and $\tw(\Fuk(M))$ the associated category of twisted complexes. One has an exact triangle of twisted complexes
\begin{equation} \label{eq:multitwist}
\xymatrix{
\tau_{V_m}^{-1}\cdots \tau_{V_3}^{-1}\tau_{V_2}^{-1}(V_1) \ar[rr]^{q} && V_1 \ar[dl] \\ & \mathrm{Cone}(q) \ar[ul]^{[1]} &
}
\end{equation}
Suppose that $\HW^*(\Delta_1,\Delta_1)$ is nonzero. A geometric argument involving the first stage of wrapping \cite[Section 5]{maydanskiy-seidel09} shows that the map $q$ in \eqref{eq:multitwist} must vanish, hence $V_1$ must be a direct  summand of $\mathrm{Cone}(q)$. To write down that cone object explicitly, we take the full $A_\infty$-subcategory of $\Fuk(M)$ with objects $V_i$, and replace it by a quasi-isomorphic strictly unital $A_\infty$-category, denoted by $\scrB$. A repeated application of the algebraic expression for a Dehn twist \cite[Corollary 17.17]{seidel04} shows that $\mathrm{Cone}(q)$ is isomorphic to the image under the quasi-isomorphic embedding $\tw(\scrB) \rightarrow \tw(\Fuk(M))$ of the twisted complex
\begin{equation} \label{eq:twisted}
C = \bigoplus \hom_{\scrB}(V_{i_1},V_{i_2})^\vee[-1] \otimes \cdots \otimes \hom_{\scrB}(V_{i_{r-1}},V_{i_r})^\vee[-1] \otimes V_{i_r}[1],
\end{equation}
where the direct sum ranges over all terms as in \eqref{eq:ineq}, and the whole comes equipped with a suitable differential $\partial_C$. We will not describe the differential completely, but the following properties are important. All nonzero terms of the differential preserve or increase $r$. Moreover, those that preserve $r$ just consist of applying the dual of $\mu^1_{\scrB}$ to one of the morphism groups in \eqref{eq:twisted}, tensored with the identity on all the other factors as well as on $V_{i_r}$.

A Lagrangian submanifold $W$ as in \eqref{eq:ineq} may not be an object of $\Fuk(M)$, since we have not required it to be closed. Still, it defines a cohomologically unital left $A_\infty$-module over $\Fuk(M)$, hence by restriction an $A_\infty$-module $\scrW$ over $\scrB$. Recall that such a module is a cohomologically unital $A_\infty$-functor $\scrB \rightarrow \Ch$ into the dg category of $\bZ/2$-graded chain complexes of $\bK$-vector spaces. On the cohomological level, this simply associates to each $V_i$ the Floer cohomology $\HF^*(W,V_i)$. Extend the functor to $\tw(\scrB) \rightarrow \Ch$ in the essentially unique way, still denoting that by $\scrW$. To compute the cohomology of $\scrW(C)$, one can use the filtration by $r$. In view of the properties of $\partial_C$ mentioned above, the resulting spectral sequence starts with
\begin{equation} \label{eq:e1}
\bigoplus \HF^*(V_{i_1},V_{i_2})[-1]^\vee \otimes \cdots \otimes \HF^*(V_{i_{r-1}},V_{i_r})[-1]^\vee \otimes \HF^*(W,V_{i_r})
\end{equation}
and converges to $H(\scrW(C))$ after finitely many steps. By assumption, $\HF^*(W,V_1)$ is a direct summand of $H(\scrW(C))$, hence its total rank cannot be more than that of \eqref{eq:e1}. An application of Poincar{\'e} duality in Lagrangian Floer theory translates this into the original statement \eqref{eq:ineq}.

\begin{remark}
While \cite{maydanskiy-seidel09} considers only Floer cohomology with coefficients in $\bK = \bF_2$, the part that is relevant here works over an arbitrary coefficient field. Note also that the definition of wrapped Floer cohomology in \cite{maydanskiy-seidel09} is not quite the standard one, being instead adapted to the special case of Lefschetz thimbles. However, the two notions are equivalent, as explained in \cite[Remark 3.1]{maydanskiy-seidel09} (and in fact, a somewhat weaker statement would be sufficient for our purpose). Finally, we should mention that there is an alternative argument, which derives the same conclusion from \eqref{eq:algebraic-hw} by algebraic manipulations.
\end{remark}

\subsection{Algebraic computations in wrapped Floer cohomology}
We next turn to Property \ref{th:iterate}, first concentrating on the purely algebraic aspects, for which we use the language and results of \cite{seidel08}. As before, $\scrB$ is a strictly unital $A_\infty$-category which is quasi-isomorphic to the full $A_\infty$-subcategory of $\Fuk(M)$ with objects $(V_1,\dots,V_m)$. By definition, $\scrB$ is a $\bZ/2$-graded $A_\infty$-category over $\bK$ with $m$ ordered objects. Equivalently, one can view it as an $A_\infty$-algebra over the semisimple ring
\begin{equation} \label{eq:semisimple}
R = \bK e_1 \oplus \cdots \oplus \bK e_m, \;\; e_i^2 = e_i, \;\; e_i e_j = 0 \text{ for $i \neq j$.}
\end{equation}
Then, strict unitality means that $\scrB$ contains a copy of $R$, embedded in a way compatible with the $\bZ/2$-grading and $R$-bimodule structure, and such that
\begin{equation}
\begin{aligned}
& \mu^1_\scrB(e_i) = 0, \\
& \mu^2_\scrB(e_i,x) = (-1)^{|x|} e_ix, \quad \mu^2_\scrB(x,e_i) = xe_i, \\
& \mu^d_\scrB(\cdots,e_i,\cdots) = 0 \text{ for all $d > 2$.}
\end{aligned}
\end{equation}
We also have the directed $A_\infty$-subcategory $\scrA \subset \scrB$, which as an algebra can be written as $\scrA = R \oplus \bigoplus_{i<j} e_j\scrB e_i$ (note that the cohomology of $e_j\scrB e_i$ is $\HF^*(V_i,V_j)$ by definition). Finally, we have the curved $A_\infty$-algebra $\scrD = \scrA \oplus t\scrB[[t]] \subset \scrB[[t]]$ with added curvature term $\mu^0_{\scrD} = t(e_1 + \cdots + e_m)$, as introduced in \cite{seidel06}.

We want to consider the associated categories of $A_\infty$-modules. More precisely, let $\mod(\scrA)$ be the $A_\infty$-category (in fact, differential graded category) of finite-dimensional strictly unital right $\scrA$-modules. This is quasi-equivalent to $\tw(\scrA)$, in a way which sends each object $V_i$ to the projective $A_\infty$-module $e_i\scrA$ \cite[Corollary 5.26]{seidel04}. For $\scrD$, we have an analogous category $\modt(\scrD)$ of finite-dimensional strictly unital torsion $A_\infty$-modules \cite[Section 4]{seidel08}. By definition, an object $\scrN$ of $\modt(\scrD)$ comes with a finite decreasing filtration $F^j\scrN$, which is compatible with its structure maps $\mu^{1|d}_\scrN: \scrN \otimes_R \scrD^{\otimes d} \rightarrow \scrN[1-d]$ in the sense that
\begin{equation}
\mu^{1|d}_{\scrN}(F^j\scrN \otimes t^{k_d}\scrD \otimes \cdots \otimes t^{k_1}\scrD) \subset
F^{j+k_1 + \cdots + k_d}\scrN
\end{equation}
(the definition of morphisms in $\modt(\scrD)$ also has a slight twist \cite[Equation (4.7)]{seidel08}, which takes the $t$-adic topology of $\scrD$ into account). The projection $\scrD \rightarrow \scrD/t\scrB[[t]] = \scrA$ induces a pullback functor $\Pi: \mod(\scrA) \rightarrow \modt(\scrD)$. Given any object $\scrN$ of $\modt(\scrD)$, the successive quotients $F^j\scrN/F^{j+1}\scrN$ lie in the image of $\Pi$ up to isomorphism. Hence, the pullbacks $\Pi(e_i\scrA)$ generate $\modt(\scrD)$.

\begin{lemma} \label{th:factor}
Take $\scrB$ considered as an $\scrA$-module. Then $\Pi(\scrB)$ is zero (by this, we mean that its endomorphism space is acyclic). Conversely, suppose that we have some other $A_\infty$-category $\scrC$, and an $A_\infty$-functor $\Phi: \mod(\scrA) \rightarrow \scrC$ such that $\Phi(\scrB)$ is zero. Then, $\Phi$ factors through $\Pi$, up to quasi-isomorphism of functors.
\end{lemma}

\begin{proof}
This is a reformulation of results from \cite{seidel08}. For any object $\scrM$ of $\mod(\scrA)$ we have a short exact sequence of modules
\begin{equation}
0 \rightarrow \scrM \otimes_\scrA \scrA \longrightarrow \scrM \otimes_\scrA \scrB \longrightarrow \scrM \otimes_\scrA (\scrB/\scrA) \rightarrow 0,
\end{equation}
where the left term is canonically quasi-isomorphic to $\scrM$ itself. The connecting homomorphism is a morphism $t_\scrM: \scrM \otimes_\scrA (\scrB/\scrA) \rightarrow \scrM$ of degree one, unique up to homotopy. By \cite[Theorem 4.1]{seidel08}, the image of $t_\scrM$ under $\Pi$ is a quasi-isomorphism, which implies that $\Pi(\scrM \otimes_\scrA \scrB)$ is zero. Specializing to $\scrM = \scrA$ yields the first part of our statement. Now take some $\Phi$ such that $\Phi(\scrB)$ is zero. The summands $\Phi(e_i \scrB) = \Phi(e_i \scrA \otimes_\scrA \scrB)$ of  $\Phi(\scrB)$  must also map to zero, and from there by a filtration argument, we conclude that $\Phi(\scrM \otimes_\scrA \scrB)$ vanishes for general  $\scrM$. As a consequence, the image of $t_\scrM$ under $\Phi$ is a quasi-isomorphism. From this, we conclude that if $\scrM$ has the property that $t_\scrM$ is nilpotent in the sense of \cite[Section 1]{seidel08}, then $\Phi(\scrM)$ must be zero. Since $\modt(\scrD)$ is quasi-equivalent to the quotient of $\mod(\scrA)$ by the subcategory of $\scrM$ with this nilpotence property \cite[Theorem 4.1]{seidel08}, the desired factorization follows from the general theory of (dg or $A_\infty$) categorical quotients \cite{keller99a,drinfeld02,lyubashenko-ovsienko06}.
\end{proof}

For the next step of our argument, let $\tilde{\scrB} \subset \scrB$ be the full $A_\infty$-subcategory containing only the last $m-1$ out of our $m$ original objects. This can also be thought of as an $A_\infty$-algebra over $\tilde{R} = \bK e_2 \oplus \cdots \oplus \bK e_m \subset R$. Analogously, we have $\tilde{\scrA}$ and $\tilde{\scrD}$, together with the pullback functor $\tilde{\Pi}: \mod(\tilde{\scrA}) \rightarrow \modt(\tilde{\scrD})$.

\begin{lemma} \label{th:algebraic-induction}
Suppose that both $\Pi(e_1\scrA)$ and $\tilde{\Pi}(\tilde{\scrA})$ are zero. Then, $\Pi(\scrA)$ is also zero.
\end{lemma}

\begin{proof}
Because of directedness, any module $\scrM$ over $\scrA$ sits in a canonical short exact sequence
\begin{equation}
0 \rightarrow \scrM e_1 \longrightarrow \scrM \longrightarrow \scrM / \scrM e_1 \rightarrow 0.
\end{equation}
Here, the submodule $\scrM e_1$ is quasi-isomorphic to a direct sum of shifted copies of the simple module $e_1 \scrA = e_1 \scrA e_1 = \bK e_1$. Modules with $\scrM e_1 = 0$ are the same as modules over the smaller algebra $\tilde\scrA$ (in more abstract terms, this yields a cohomologically fully faithful embedding $\mod(\tilde\scrA) \rightarrow \mod(\scrA)$, for which the projection $\scrM \mapsto \scrM / \scrM e_1$ is a right adjoint).

First take $\scrM = e_i\scrB$ for some $i>1$. We know from Lemma \ref{th:factor} that $\Pi(e_i\scrB)$ is zero, and so is $\Pi(e_i \scrB e_1)$ by assumption. It follows that $\Pi$ sends the quotient $e_i\scrB / e_i \scrB e_1$ to zero, but that quotient can be identified with $e_i\tilde\scrB$. As a consequence of this and Lemma \ref{th:factor} applied to the smaller algebra, we have a factorization up to quasi-isomorphism of functors
\begin{equation} \label{eq:composite-wrap}
\xymatrix{
\mod(\tilde{\scrA}) \ar[r] \ar[d]_-{\tilde{\Pi}} & \mod(\scrA) \ar[d]^-{\Pi} \\
\modt(\tilde{\scrD}) \ar@{-->}[r] & \modt(\scrD).
}
\end{equation}
Now consider $\scrM = e_i \scrA$ for some $i>1$. Again by assumption, $\Pi(e_i \scrA e_1)$ is zero. By the factorization established above, $\Pi(e_i \scrA / e_i \scrA e_1) = \Pi(e_i \tilde{\scrA})$ is the image of $\tilde{\Pi}(e_i \tilde{\scrA})$ under some $A_\infty$-functor, hence zero by the other part of our assumption. These two facts imply that $\Pi(e_i \scrA)$ is zero, as claimed. \end{proof}

There is a conjectural formula for wrapped Floer cohomology of Lefschetz thimbles, stated in \cite{seidel08} and proved in \cite[Appendix]{bourgeois-ekholm-eliashberg09}. This says that for any $1 \leq i \leq m$,
\begin{equation} \label{eq:algebraic-hw}
\HW^*(\Delta_i,\Delta_i) \iso H^*(\hom_{\modt(\scrD)}(\Pi(e_i\scrA), \Pi(e_i\scrA))).
\end{equation}
Hence, if $\HW^*(\Delta_1,\Delta_1) = 0$, then $\Pi(e_1\scrA)$ is zero. Of course, the corresponding statement holds for each thimble $\tilde{\Delta}_i$ as well. Hence, if $\HW^*(\tilde{\Delta}_i,\tilde{\Delta}_i) = 0$, then $\tilde{\Pi}(e_i\tilde{\scrA})$ is zero, and if this holds for all $2 \leq i \leq m$, then $\tilde{\Pi}(\tilde{\scrA})$ is zero. Under these assumptions, Lemma \ref{th:algebraic-induction} implies that $\HW^*(\Delta_i,\Delta_i) = 0$ for all $i$, which completes the proof of Property \ref{th:iterate}.

The proof of Property \ref{th:meta} relies on the same technology. Suppose that we have two different geometric situations, with the same number of vanishing cycles, and in which the associated $A_\infty$-algebras $\scrB$ are quasi-isomorphic over $R$. Without loss of generality, we may assume that the quasi-isomorphism is strictly unital \cite[Th\'eor\`eme 3.2.2.1]{lefevre}. Then, it automatically induces a quasi-isomorphism between the directed subalgebras $\scrA$, and a filtered quasi-isomorphism between the curved $A_\infty$-algebras $\scrD$. These in turn induce quasi-equivalences of the associated categories $\modt(\scrD)$, which are compatible with the pullback maps $\Pi$. In view of \eqref{eq:algebraic-hw}, this proves the desired statement.

\begin{remark}
In a paper in preparation, the authors will give a proof of \eqref{eq:algebraic-hw} which is independent of \cite{bourgeois-ekholm-eliashberg09}, under the additional assumption that the coefficient field $\bK$ is of characteristic $\neq 2$ (this technical restriction is inherited from \cite{seidel04}). Relying on this alternative approach would require that we exclude characteristic $2$ from Theorems \ref{th:1}, \ref{th:2} and \ref{th:3}. The applications in Corollaries \ref{th:1.5}, \ref{th:2.5} and \ref{th:3.5} would still be valid as stated.
\end{remark}

\subsection{The open-closed string relationship\label{subsec:bee}}
The last remaining fundamental fact stated in Section \ref{subsec:lefschetz} is Property \ref{th:open-closed}. One direction is easy and uses only general TQFT type operations, see for instance \cite[Theorem 56]{ritter10}: $\SH^*(E)$ is a unital ring, and the wrapped Floer cohomology of any exact Lagrangian submanifold with Legendrian boundary (or more generally, any pair of such submanifolds) is a module over that ring. Hence, if $\SH^*(E)$ vanishes, so do all the wrapped Floer cohomology groups.

The converse is a consequence of the surgery formula for symplectic cohomology from \cite{bourgeois-ekholm-eliashberg09}. A quick way of deriving this would be as follows. Conjectural algebraic formulae for $\SH^*(E)$ and $\HW^*(\Delta_i,\Delta_i)$ were given in \cite{seidel06,seidel08}. In this framework, it follows from \cite[Lemma 5.2]{seidel08} that if the algebraic counterpart of $\HW^*(\Delta_i,\Delta_i)$ vanishes for all $i$, then so does that of $\SH^*(E)$. Finally, these conjectures were proved in \cite[Appendix]{bourgeois-ekholm-eliashberg09}. However, this strategy is a bit roundabout, and involves more Lefschetz fibration theory than is strictly necessary. We will therefore also give a more direct argument, which stays close to \cite[Section 6.2]{bourgeois-ekholm-eliashberg09}.

For that, let us start with algebraic preliminaries. Let $R$ be the semisimple ring from \eqref{eq:semisimple}. Given any $\bZ/2$-graded $R$-bimodule $C$, one can form the associated tensor algebra $T(C)$ over $R$. More concretely, $C$ is the direct sum of $\bZ/2$-graded $\bK$-vector spaces $e_j C e_i$. If one thinks of basis elements of those vector spaces as letters, then $T(C)$ consists of composable words:
\begin{equation}
\begin{aligned}
T(C) & = R \oplus C \oplus C \otimes_R C \oplus \cdots \\ & = R \oplus \bigoplus_{i,j} e_j C e_i \oplus
\bigoplus_{i,j,k} e_k C e_j \otimes_\bK e_j C e_i \oplus \cdots
\end{aligned}
\end{equation}
Suppose that we are given a differential $\delta$ which makes $T(C)$ into a unital differential graded algebra over $R$. Such a differential is uniquely determined by its behaviour on generators, which is an $R$-bimodule map $C \rightarrow T(C)$. In particular, we can then consider the diagonal part
\begin{equation} \label{eq:diagonal-part}
T(C)^{\mathrm{diag}} = \bigoplus_i e_i T(C) e_i = R \oplus \bigoplus_i e_i C e_i \oplus \bigoplus_{i,j} e_i C e_j \otimes_{\bK} e_j C e_i \oplus \cdots,
\end{equation}
which is a direct summand of $T(C)$ as a chain complex. Next, take $\Omega(C) = (C[1] \otimes_R T(C))^{\mathrm{diag}}$, where the diagonal part is defined as before. There is an induced differential $\delta^{\mathrm{cycl}}$ on this, well-known in the algebra literature (see \cite{quillen88} or \cite[Section 7.2]{ks}). To define that, it is convenient to think of the generators of $\Omega(C)$ as formal noncommutative one-forms $c_1\, \cdots c_{i-1} \, dc_i \, c_{i+1} \cdots c_r$, with the rule that any two such expressions which are related by a cyclic permutation are the same (up to the standard Koszul signs). Then, $\delta^{\mathrm{cycl}}$ is formally identical to the standard way of applying the Lie derivative of a vector field to a one-form. One can form a total complex
\begin{equation} \label{eq:total-complex}
\left(T(C)^{\mathrm{diag}} \oplus \Omega(C), \begin{pmatrix} \delta & S \\ 0 & \delta^{\mathrm{cycl}} \end{pmatrix} \right),
\end{equation}
where $S(dc_1\, c_2 \cdots c_r) = c_1 c_2 \cdots c_r - (-1)^{|c_1|(|c_2| + \cdots + |c_r|)} c_2 \cdots c_r c_1$. Next, suppose that $C$ itself comes with an exhausting increasing filtration $0 = F_0C \subset F_1C \subset \cdots$, such that $\delta$ is strictly decreasing with respect to the induced filtration of $T(C)$.

\begin{lemma} \label{th:filter}
If $T(C)^{\mathrm{diag}}$ is acyclic, then so is \eqref{eq:total-complex}.
\end{lemma}

\begin{proof}
Clearly, we can quotient out by $T(C)^{\mathrm{diag}}$, hence only need to show that $\Omega(C)$ is acyclic.
Equip $\Omega(C)$ with the filtration $F_p \Omega(C) = (F_p C[1] \otimes T(C))^{\mathrm{diag}}$, which is compatible with $\delta^{\mathrm{cycl}}$ by definition. Then, each associated graded space is
\begin{equation} \label{eq:quo}
F_p \Omega(C)/F_{p-1}\Omega(C) = \bigoplus_{i,j} e_j F_p C[1]/F_{p-1} C[1] e_i \otimes_{\bK} e_i T(C) e_j,
\end{equation}
with a differential which is $\mathrm{id} \otimes \delta$. Because of the algebra structure of $T(C)$, acyclicity of its diagonal part implies the acyclicity of the whole, hence of $e_i T(C) e_j$ for any $(i,j)$. Hence, all quotients \eqref{eq:quo} are acyclic, and so is $\Omega(C)$ itself.
\end{proof}

To explain the geometric relevance of this, we have to consider the following instance of Weinstein handle attachment. Suppose that we have a Liouville domain $E^0$ with $\SH^*(E^0) = 0$, as well as a collection of $m$ disjoint Legendrian spheres $(K_1^0,\dots,K_m^0)$ in $\partial E^0$. Let $E^1$ be the Liouville domain obtained by attaching a Weinstein handle to each sphere. This contains canonical Lagrangian discs with Legendrian boundary, the co-cores $(W^1_1,\dots,W^1_m)$ (see Section \ref{subsec:weinstein} below, and references there, for a more detailed discussion). Assuming suitably generic properties of the Reeb flow, let $e_j C e_i$ be the space freely generated by Reeb chords joining $K^0_i$ to $K^0_j$; any choice of orientation on the spheres determines a $\bZ/2$-grading by intersection number. Then, the differential graded algebra $(T(C),\delta)$ defined above is $LHA(K^0)$ in the notation from \cite{bourgeois-ekholm-eliashberg09}, where $K^0 = K^0_1 \cup \cdots \cup K^0_m$; while \eqref{eq:total-complex} is $LH^{\mathit{Ho}}(K^0)$. The cohomology of the latter complex is $\SH^*(E^1)$ by \cite[Corollary 5.7]{bourgeois-ekholm-eliashberg09}. There is also an open string analogue of the statement, which is a minor variation on \cite[Theorem 5.8]{bourgeois-ekholm-eliashberg09}, and which says that the cohomology of $e_i T(C) e_i$ is the wrapped Floer cohomology $\HW^*(W^1_i,W^1_i)$ in $E^1$. Finally, $C$ carries a filtration by action, which is indexed by reals but only jumps at a discrete set of values. The argument from Lemma \ref{th:filter} still goes through in this slightly modified algebraic framework. The conclusion is that the vanishing of $\HW^*(W^1_i,W^1_i)$ for all $i$ implies that of $\SH^*(E^1)$ (compare \cite[Section 6.2]{bourgeois-ekholm-eliashberg09}, which discusses the case of a single handle attachment). To apply this to our situation, we need to recall that forming the total space of a Lefschetz fibration can be thought of as a special case of handle attachment, where $E^0$ is a version of $D^2 \times M$ with the corners rounded off, the $K^0_i$ arise from the vanishing cycles, $E^1 = E$, and $W^1_i = \Delta_i$ are the Lefschetz thimbles (see Remark \ref{th:lefschetz-handles}, or \cite[Section 7]{bourgeois-ekholm-eliashberg09} for a more detailed discussion). Since the symplectic cohomology of $E^0$ vanishes by \cite{oancea04}, Property \ref{th:open-closed} follows.

\begin{remark}
In forthcoming work, the authors will give a proof of Property \ref{th:open-closed} by a ``decomposition of the diagonal'' argument, which is independent of \cite{bourgeois-ekholm-eliashberg09}.
\end{remark}

\subsection{A digression\label{subsec:poly}}
We need to explain a technical point which was tacitly used in our proof of Property \ref{th:open-closed}. Namely, even though all statements in \cite{bourgeois-ekholm-eliashberg09} are originally formulated for $\bK = \bQ$, the specific parts we require work for arbitrary $\bK$. This is not really surprising (compare for instance \cite{ekholm08}, which introduced rational Symplectic Field Theory for Lagrangian submanifolds with coefficients in $\bK = \bF_2$), and only requires careful inspection of the moduli spaces involved. As a representative sample, we consider the definition of reduced symplectic cohomology $\SH^*_+(M)$ in an SFT framework, following \cite[Sections 2 and 3]{bourgeois-ekholm-eliashberg09} (see also \cite{bourgeois-oancea09}). The other moduli spaces relevant for our purpose (the ones underlying the construction of full symplectic cohomology, Legendrian contact homology, and the isomorphism in \cite[Corollary 5.7]{bourgeois-ekholm-eliashberg09}) can be treated analogously.

Fix a finite type Liouville manifold $M$. Let $N$ be its boundary at infinity, which in the terminology of Section \ref{subsec:def} means a large level set $h^{-1}(c)$, $c \gg 0$, with its contact one-form. At infinity, $M$ is isomorphic to the positive half $\bR^+ \times N$ of the symplectization of $N$. Equip the symplectization with an almost complex structure in the standard translation-invariant way (see for instance \cite[Section 1.4]{sft-intro}). Then, equip $M$ with an almost complex structure which, at infinity, agrees with the one on the symplectization.

Without essential loss of generality, we may assume that all periodic orbits of the Reeb flow on $N$ are transversally nondegenerate. Given two such orbits (possibly multiply covered, but unparametrized) $\gamma_+$ and $\gamma_-$, we consider the moduli space $\scrMan(\gamma_-,\gamma_+)$ of {\em anchored holomorphic cylinders}, following \cite[Section 2.3]{bourgeois-ekholm-eliashberg09}. Points of these moduli spaces are represented by the following kind of structure:
\begin{itemize}
\item
Take a nodal genus zero Riemann surface $\bar{C}$ with two (smooth and distinct) marked points $\bar{z}_{0,+},\bar{z}_{0,-}$. We denote by $C$ the normalization of $\bar{C}$. By definition, special points of $C$ are the preimages of our marked points, denoted by $z_{0,\pm}$, and preimages of the nodes, denoted by $z_{i,\pm}$ for $i = 1,\dots,k$ (the labeling rule is that, of the two preimages of any given node, $z_{i,+}$ is the one lying on the component that is combinatorially further away from $z_{0,+}$; see Figure \ref{fig:curve1}). The complement of the special points is written as $C^{\circ} \subset C$. Additionally, for each $z_{i,+}$, $0 \leq i \leq k$, we want to have an asymptotic marker \cite[Section 1.5]{sft-intro}, which means an isomorphism $TC_{z_{i,+}} \iso \bC$, unique up to multiplication with $\bR^+$.
\end{itemize}
\begin{figure}
\begin{center}
\epsfig{file=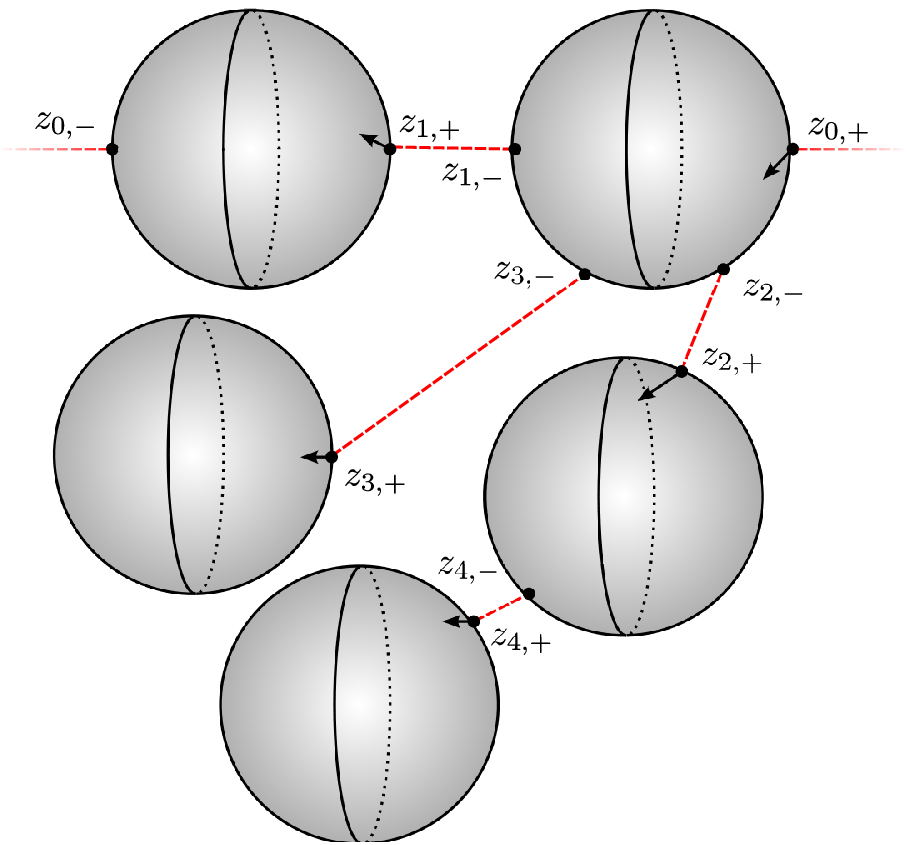}
\caption{\label{fig:curve1}}
\end{center}
\end{figure}
Note that if we have a smooth genus zero Riemann surface with two marked points, that surface can be identified with the Riemann sphere by sending one point to $0$ and the other to $\infty$; in particular an asymptotic marker at one point determines one at the other. In our situation, we are given an asymptotic marker at exactly one point of each connected component of $C$, and that therefore induces markers at the other points as well.
\begin{itemize} \itemsep1em
\item
To each marked point $z_{i,\pm}$ we associate a Reeb orbit $\gamma_{i,\pm}$, as follows. $\gamma_{0,\pm} = \gamma_{\pm}$ are the given orbits. For $i>0$, we require that $\gamma_{i,+} = \gamma_{i,-}$.
\end{itemize}
At this point, still following \cite[Section 2.3]{bourgeois-ekholm-eliashberg09}, a small but significant difference from the classical SFT setup occurs. Namely, we do not choose preferred points on any of the Reeb chords.
\begin{itemize}
\item
Denote by $\bar{M}$ the disjoint union of $M$ and infinitely many copies of $\bR \times N$.
We then want to have a map $u: C^{\circ} \rightarrow \bar{M}$ which is a pseudo-holomorphic building in the sense of \cite[Section 9]{sft-compactness}, with asymptotic limits $\gamma_{i,\pm}$ at the points $z_{i,\pm}$ (Figure \ref{fig:building}). For $i>0$, we require that the parametrizations of the orbits $\gamma_{i,\pm}$ inherited from our asymptotic markers should coincide.
\end{itemize}
As usual, there is a stability condition, which places some restrictions on trivial cylinders. Two holomorphic buildings define the same point of $\scrMan(\gamma_-,\gamma_+)$ if they are identified by a combination of the following two relations. The first one is isomorphism of the domains, compatible with the asymptotic markers. The second is translation on each $\bR \times N$ component of the target space. As a consequence, the moduli space comes with evaluation maps
\begin{equation} \label{eq:ev}
\scrMan(\gamma_-,\gamma_+) \longrightarrow \underline{\gamma}_- \times \underline{\gamma}_+,
\end{equation}
where $\underline{\gamma}_\pm$ are the geometric images (which are copies of the underlying simple orbits) of $\gamma_\pm$.
\begin{figure}
\begin{center}
\epsfig{file=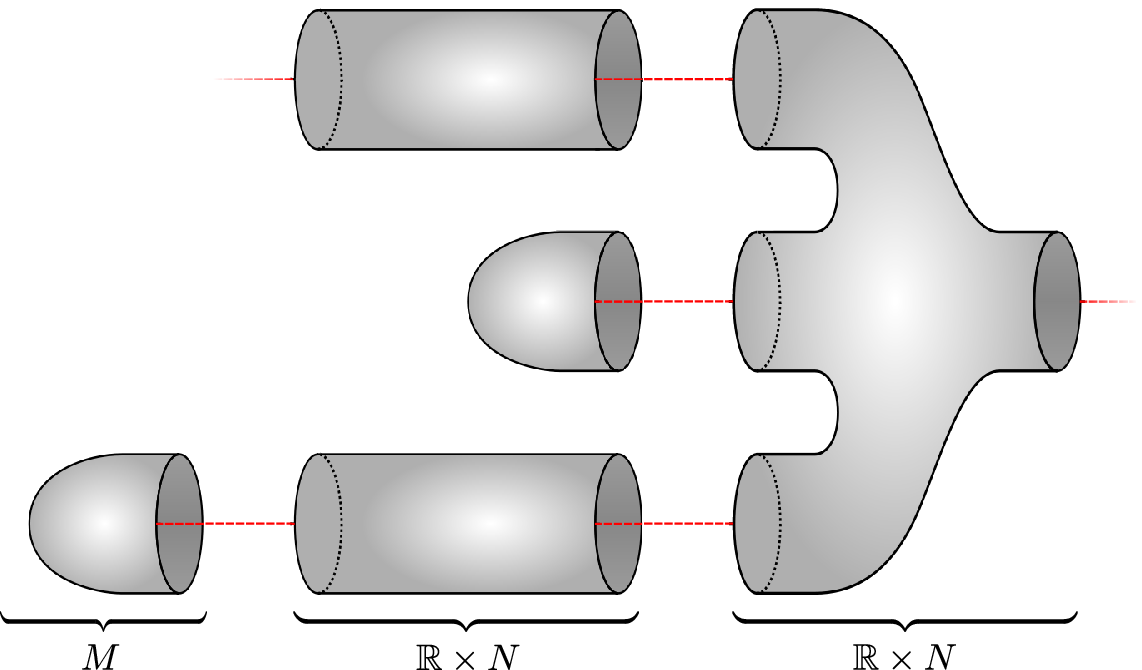}
\caption{\label{fig:building}}
\end{center}
\end{figure}
As part of the general definition, each point of the moduli space also carries an automorphism (or isotropy) group. Stability ensures that this is finite, but in our case it is actually always trivial. This is because, given a sphere with a marked point and an asymptotic marker at that point, there are no nontrivial finite order automorphisms which preserve that marker. Pre-gluing of different components of our pseudoholomorphic maps must be compatible with the asymptotic markers by definition, hence there is no ambiguity in gluing the ends together. These two facts together ensure that $\scrMan(\gamma_-,\gamma_+)$ can be given the structure of an M-polyfold, in the sense of \cite{poly}. Note that in this framework, the ``boundary strata'' of our moduli spaces are fibre products rather than products, meaning that for any $\gamma$ we have a canonical embedding
\begin{equation}
\scrMan(\gamma_-,\gamma) \times_{\underline{\gamma}} \scrMan(\gamma,\gamma_+) \hookrightarrow
\scrMan(\gamma_-,\gamma_+)
\end{equation}
where the fibre product is taken with respect to \eqref{eq:ev}. This is formally the situation of any Morse-Bott type theory. By taking suitable preimages under the evaluation maps, one can define subspaces of our moduli spaces, and (applying the standard machinery) obtain point-counting numbers which can be used to define a chain complex $\mathit{SC}^*_+(M)$, with two generators for each Reeb orbit (including ``bad orbits''), a process which is described in detail in \cite[Section 3.2]{bourgeois-ekholm-eliashberg09}. In the M-polyfold context, these numbers are integers, hence the chain complex can be defined with coefficients in any field (or indeed, abelian group). The cohomology of this complex is then $\mathit{SH}^*_+(M)$.

\subsection{Almost symplectic geometry\label{subsec:almost}}
This is an elementary postscript to our discussion of Picard-Lefschetz theory, in which we recall its homotopy theoretic analogue. First, an {\em almost symplectic manifold} is a manifold $M^{2n}$ together with a two-form $\omega_M$ which is pointwise nondegenerate, but not necessarily closed (this is the same as an $\mathit{Sp}(2n,\bR)$ structure on the abstract vector bundle $TM \rightarrow M$).

An {\em almost symplectomorphism} between two such manifolds is a diffeomorphism $\phi: M \rightarrow \tilde{M}$ together with a one-parameter family $(\omega_t)_{0 \leq t \leq 1}$ of almost symplectic structures, which interpolate between $\omega_0 = \omega_M$ and $\omega_1 = \phi^*\omega_{\tilde{M}}$. Similarly, an {\em almost Lagrangian submanifold} of an almost symplectic manifold $M$ is an $n$-dimensional submanifold $L$ together with a one-parameter family $(\omega_t)$ as before, such that $\omega_0 = \omega_M$ and $\omega_1|L = 0$ (equivalently up to homotopy, one could describe this structure by giving an $n$-dimensional subbundle of $TM|L \rightarrow L \times [0,1]$ which equals $TL$ over $L \times \{0\}$ and is $\omega_M$-isotropic over $L \times \{1\}$). Note that both notions are isotopy invariant. Namely, if $\phi$ is an almost symplectomorphism and $\tilde{\phi}$ is another diffeomorphism isotopic to $\phi$, the isotopy equips $\tilde{\phi}$ with an almost symplectomorphism structure. Similarly, if $L$ is almost Lagrangian and $\tilde{L}$ is isotopic to it, the isotopy equips $\tilde{L}$ with an almost Lagrangian structure (one can see this, for instance, by embedding the isotopy into one of diffeomorphisms).

Suppose that $(M,\omega_M)$ is almost symplectic, and $L \subset M$ an almost Lagrangian sphere (which should come with a diffeomorphism $S^n \rightarrow L$, unique up to the same ambiguities as in Section \ref{subsec:lefschetz}). One can then define the Dehn twist $\tau_L$ as an almost symplectomorphism of $M$, unique up to isotopy in an appropriate sense. We should recall that for $n = 2$, the square $\tau_L^2$ is isotopic to the identity through almost symplectomorphisms \cite{seidel98b}. Suppose that we are given an ordered family $(V_1,\dots,V_m)$ of almost Lagrangian spheres. One can use those as vanishing cycles and construct a $(2n+2)$-dimensional almost symplectic manifold $E$. The almost symplectic isomorphism type of $E$ depends only on the isotopy class of the vanishing cycles as almost Lagrangian spheres. Of course, in the actual symplectic case all these notions specialize to the standard ones.

\section{Altering Lefschetz fibrations}

\subsection{Weinstein handle attachment\label{subsec:weinstein}}
As announced before, we need to consider the construction of Weinstein handles \cite{weinstein91} of critical dimension in a little more detail. Let $M^0$ be a Liouville domain of dimension $2n \geq 4$, and $K^0 \subset \partial M^0$ a Legendrian sphere. Attaching a Weinstein handle to $K^0$ yields a larger Liouville domain $M^1 \supset M^0$. The new boundary $\partial M^1$ always contains another Legendrian sphere $K^1$, which is in fact the boundary of a Lagrangian disc $W^1 \subset M^1 \setminus M^0$, the co-core of the handle. Of particular interest for us is the case when the original Legendrian sphere $K^0$ already comes as the boundary of a Lagrangian disc $W^0 \subset M^0$. Then, $W^0$ and the core of the handle can be glued together to form a Lagrangian homotopy sphere $L^1 \subset M^1$, which intersects $W^1$ transversely in a single point (see Figure \ref{fig:handle} for a schematic description).
\begin{figure}
\begin{center}
\epsfig{file=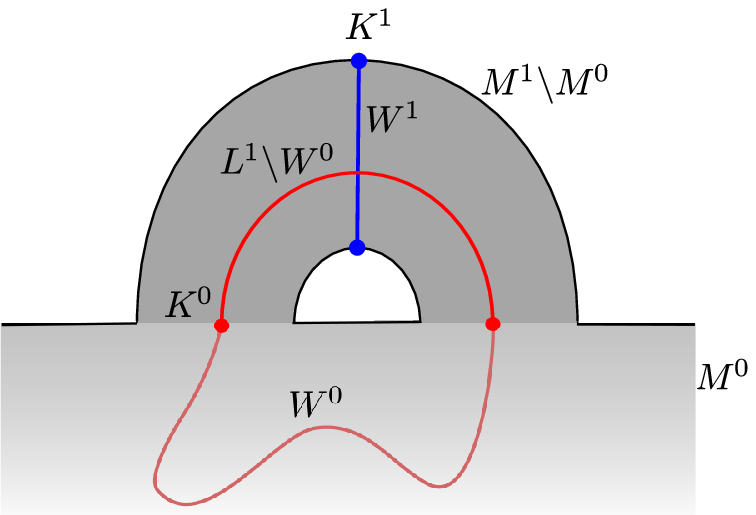}
\caption{\label{fig:handle}}
\end{center}
\end{figure}

In the simple form just described, the handle attachment process is not quite unique. To fix this, we want our attaching sphere to come with a diffeomorphism $S^{n-1} \rightarrow K^0$, unique up to isotopy and composition with elements of $O(n)$, just as for vanishing cycles. This diffeomorphism determines the handle attachment up to deformation within the class of Liouville domains. Note that the co-core disc $W^1$ always comes with a diffeomorphism $D^n \rightarrow W^1$, unique up to the same ambiguity as before (the additional information contained in this is trivial except possibly in dimensions $n = 4,5$, since $\Diff^+(D^n)$ is connected for all other $n$ \cite{smale59, cerf68, cerf70}). By restriction, we get a diffeomorphism $S^{n-1} \rightarrow K^1$ (this allows one to iterate the handle attachment process in a natural way). In the case where $K^0 = \partial W^0$, we can similarly assume that our given diffeomorphism $S^{n-1} \rightarrow K^0$ is the restriction of a diffeomorphism $D^n \rightarrow W^0$. That choice ensures that the resulting $L^1 \subset M^1$ is an actual differentiable sphere. In fact, it will then come with a diffeomorphism $S^n \rightarrow L^1$, unique in the same sense as before. From now on, whenever we discuss handle attachment, these additional diffeomorphism data are assumed to be present.

Weinstein handles appear in the context of stabilization of open book decompositions \cite{giroux02}, and also in the related argument for Lefschetz fibrations. Suppose as before that $M^0$ contains a Lagrangian disc $W^0$, and let $M^1$ be the outcome of the associated handle attachment, with its Lagrangian sphere $L^1$. We then have the following folk stabilization theorem:

\begin{lemma} \label{th:cancellation}
Take any collection of Lagrangian spheres $V^0_i$, $1 \leq i \leq m$, in $M^0$. The total space of the Lefschetz fibration with fibre $M^0$ and vanishing cycles $(V^0_1,\dots,V^0_m)$ is deformation equivalent, as a Liouville domain, to that of the fibration with fibre $M^1$ and vanishing cycles $(V^1_1,\dots,V^1_{m+1})$, where
\begin{equation}
V^1_i = \begin{cases} V^0_i & i < j, \\ L^1 & i = j, \\ V^0_{i-1} & i > j \end{cases}
\end{equation}
for an arbitrary choice of $1 \leq j \leq m+1$. \qed
\end{lemma}

We complete our discussion by presenting a slightly different viewpoint. First, as a somewhat trivial example of the general class of Morse-Bott handle attachments \cite{johns09}, we have the following process. Take two Liouville domains $M^0_+$ and $M^0_-$, together with Legendrian embeddings $K \hookrightarrow \partial M^0_\pm$ of the same closed manifold into their respective boundaries, with image $K^0_\pm$. One can then form a Liouville domain $M^1 \supset M^0_+ \cup M^0_-$ by attaching a handle with core $K \times [-1,1]$. The actual form of the handle is
\begin{equation} \label{eq:morse-bott-handle}
\begin{aligned}
& H = \{ (x,s,t) \in T^*K \times \bR \times [-1,1] \;:\;
\|x\|^2 + s^2 \leq \epsilon \}, \\
& \theta_H = \theta_{T^*K} + 2 s\, dt + t\, ds.
\end{aligned}
\end{equation}
Here, $\theta_{T^*K}$ is the canonical one-form on the cotangent bundle, $\|\cdot\|$ is the norm obtained from some Riemannian metric on $K$, and $\epsilon>0$ a small number. $H$ is a manifold with corners, whose boundary is the union of $\partial_{\mathrm{out}} H = \{\|x\|^2 + s^2 = \epsilon\}$ and $\partial_{\mathrm{in}} H = \{t = \pm 1\}$. The Liouville vector field
\begin{equation}
Z_H = Z_{T^*K} + 2s \partial_s - t \partial_t
\end{equation}
points strictly outwards along $\partial_{\mathrm{out}} H$, and strictly inwards along $\partial_{\mathrm{in}} H$. Using the local normal form theorem for Legendrian submanifolds inside contact type hypersurfaces \cite[Proposition 4.2]{weinstein91}, we glue together $M^0_\pm$ and $H$, identifying the two parts of $\partial_{\mathrm{in}} H$ with neighbourhoods of $K^0_{\pm} \subset \partial M^0_{\pm}$. The result at this point is a symplectic manifold with concave codimension two corners, and which can be equipped with a Liouville vector field that points strictly outwards everywhere. It is then easy to smooth out the corners while preserving the last-mentioned property.

Weinstein's construction can be thought of as a special case of the one we have just described, as follows. Given a single Legendrian sphere $K^0_+ \subset \partial M^0_+$, we take $M^0_- = D^{2n}$ to be the unit disc with its standard Liouville structure (meaning that the Liouville vector field is radial), and $K^0_-$ the flat sphere $\{0\}^n \times S^{n-1} \subset \partial M^0_-$. Attaching \eqref{eq:morse-bott-handle} with $K = S^{n-1}$ is then the same as Weinstein handle attachment. In this picture (see Figure \ref{fig:handle2}), the co-core disc is $D^n \times \{0\}$. If we assume that $K^0_+$ is the boundary of some Lagrangian disc in $M^0_+$, the resulting sphere $L_1$ is given by the union of that disc with $\{0\} \times D^n \subset D^{2n}$ and $S^{n-1} \times [-1,1] \subset H$. This is of course not all that different from the original description, but it will prove convenient for us.
\begin{figure}
\begin{center}
\epsfig{file=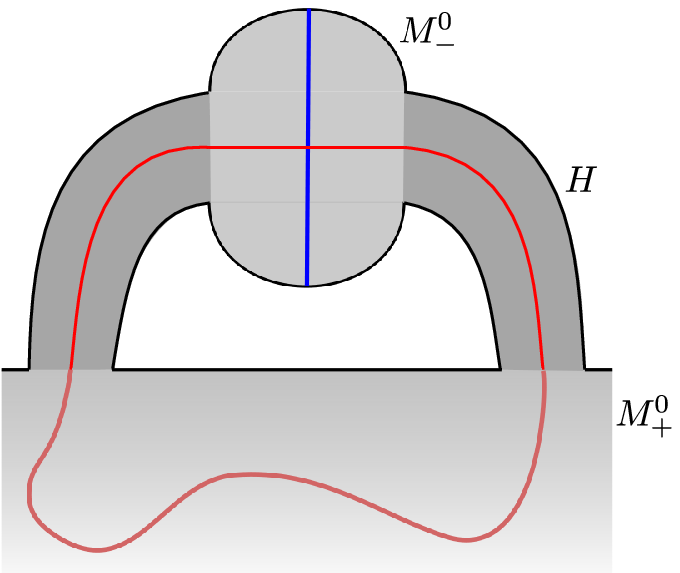}
\caption{\label{fig:handle2}}
\end{center}
\end{figure}

\subsection{A double handle attachment}
Start with a decomposition of the sphere $S^{n-1}$ into two (not necessarily connected) codimension zero submanifolds $U_\pm$, which intersect exactly along their common boundary $\partial U_+ = \partial U_-$, and of which one has Euler characteristic
\begin{equation} \label{eq:euler}
\chi(U_-) = 1.
\end{equation}
We can find a smooth $C^1$-small function $g: S^{n-1} \rightarrow \bR$ having $0$ as a regular value, with $g^{-1}(0) = \partial U_+ \cap \partial U_-$, and which is strictly positive on the interior of $U_+$, as well as strictly negative on the interior of $U_-$. Extend this to a smooth function
\begin{equation} \label{eq:extension}
g: \bR^n \longrightarrow \bR,
\end{equation}
small in the interior of the unit disc, and satisfying $g(tq) = t^2 g(q)$ for all $|q| \geq 1/2$ and $t \geq 1$. This gives rise to a Lagrangian submanifold
\begin{equation} \label{eq:graph}
G = \mathrm{Graph}(dg) \cap D^{2n} \subset D^{2n}.
\end{equation}
Here, $\mathrm{Graph}(dg) = \{p = dg_q\}$ is thought of as lying inside $T^*\bR^n = \bR^{2n}$, and we then intersect it with the unit disc. $G$ intersects the boundary in a Legendrian submanifold. To see that, note that the homogeneity assumption implies that $D^2 g_q(q,\cdot) = dg_q$, which in turn implies that on the region where $|q| \leq 1/2$, the radial vector field $Z_{\bR^{2n}}$ is tangent to the graph of $dg$. On the other hand, by choosing $g$ small we may assume that for all $|q| \leq 1/2$, $\|dg_q\|^2 < 3/4$, which means that the relevant part of the graph of $dg$ stays inside the unit ball. Hence, along its boundary $G$ is tangent to the radial vector field, which implies the desired condition. The same argument shows that $G$ itself projects to a star-shaped region with smooth boundary in $\{0\}^n \times \bR^n$, hence is a disc. Finally, homogeneity and the regularity assumption on the zero level set imply that $g$ has no critical points outside the region $|q| < 1/2$. From this, we conclude that $\partial G$ is disjoint from the standard sphere $\{0\}^n \times S^{n-1}$.

\begin{lemma} \label{th:disjoin}
By a (smooth but not usually Lagrangian) isotopy inside $D^{2n}$ which leaves $\partial G$ fixed, one can make $G$ disjoint from the zero-section.
\end{lemma}

\proof The degree of $dg$ is (up to sign) the Euler characteristic of $\bR^{2n}$ relative to a sufficiently negative level set of $g$, which is the Euler characteristic of the pair $(D^{2n},U_-)$, hence zero by assumption \eqref{eq:euler}. One can therefore find a family $(X_t)_{0 \leq t \leq 1}$ of maps $\bR^n \rightarrow \bR^n$, such that $X_0 = dg$, $\partial_t X_t = 0$ outside a compact subset, and $X_1$ is nowhere vanishing. The graphs $\mathrm{Graph}(X_t)$ form a compactly supported isotopy which displaces $\mathrm{Graph}(dg)$ from the zero-section. Using the homogeneity property, one can easily shrink this so that it is supported inside the unit disc. \qed
\begin{figure}
\begin{center}
\epsfig{file=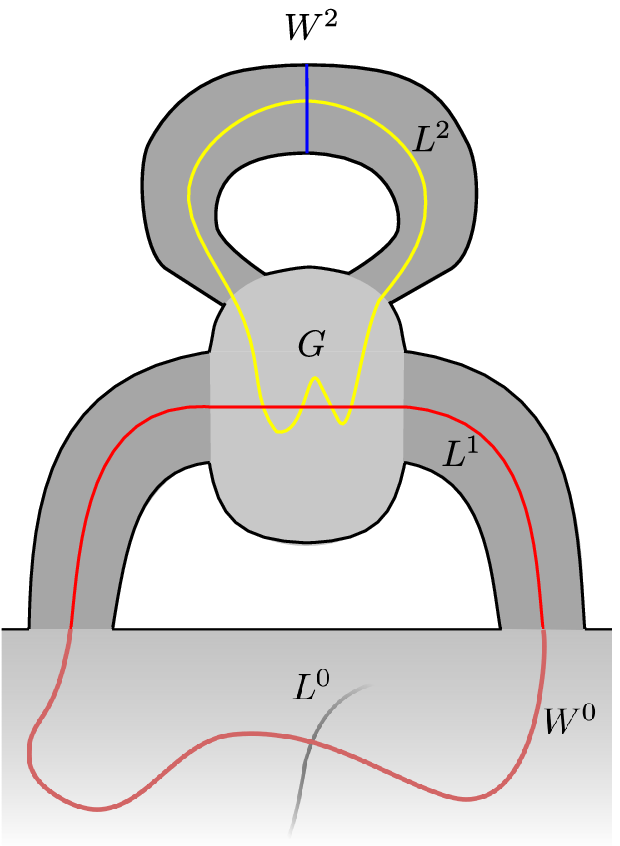}
\caption{\label{fig:doublehandle}}
\end{center}
\end{figure}

We now describe our main construction. Start with a Liouville domain $M^0$ containing both a Lagrangian sphere $L^0$ and a Lagrangian disc $W^0$ with Legendrian boundary, such that $L^0$ and $W^0$ intersect transversely and in a single point. Attach a Weinstein handle to $K^0 = \partial W^0$, forming $M^1$. We want to follow the alternative description of handle attachment indicated above, according to which $M^1$ contains a copy of $D^{2n}$. Since $\partial G$ is disjoint from the submanifold $\{0\} \times S^{n-1}$ near which we attach \eqref{eq:morse-bott-handle}, it survives into $\partial M^1$ if the size $\epsilon$ of the handle is chosen to be sufficiently small. In this way, we get another Lagrangian disc $W^1 = G$ in $M^1$ with boundary $\partial W^1 = \partial G$. Attach a second Weinstein handle to this, and denote the result by $M^2$. By construction, this contains a Lagrangian sphere $L^2$ obtained by gluing together $W^1$ and the core disc of the second handle. Finally, we have the co-core disc of that handle, which we denote by $W^2 \subset M^2$ (Figure \ref{fig:doublehandle} gives an overview of the situation). Because the intersections are so simple, it is easy to show that
\begin{align}
& \HF^*(L^1,L^0) \iso \bK, \label{eq:hf10} \\
& \HF^*(L^2,L^0) = 0, \\
& \HF^*(W^2,L^0) = 0, \\
& \HF^*(W^2,L^1) = 0, \label{eq:zerod} \\
& \HF^*(W^2,L^2) \iso \bK. \label{eq:oned}
\end{align}
Strictly speaking, to get $\bZ/2$-graded Floer cohomology groups, we need to fix orientations of all our Lagrangian submanifolds. Supposing that an orientation of $L^0$ is given, let us orient $L^1$ so that \eqref{eq:hf10} is nontrivial in odd degree. $G$ inherits a standard orientation as graph, and we equip $L^2$ with the induced orientation. Finally, $W^2$ is oriented so that \eqref{eq:oned} is nontrivial in even degree.

\begin{lemma} \label{th:classical-u}
$\HF^*(L^2,L^1) \iso \tilde{H}^{*-1}(U_-;\bK)$, where the right hand side is reduced (ordinary) cohomology.
\end{lemma}

\begin{proof}
By rescaling the function $g$ appropriately, one can achieve (without changing $\partial G$ and hence the rest of the construction) that the holomorphic strips involved in computing this Floer cohomology group have very small energy. By Monotonicity Lemma arguments, this implies that all these strips are actually contained in $D^{2n}$, so we can carry out the computation there. By the same argument, the Floer cohomology of $(G,\{0\}^n \times D^n)$ inside $D^{2n}$ is the same as that of $(\mathrm{Graph}(dg),\{0\}^n \times \bR^n)$ inside $\bR^{2n}$. Finally, by a variant of Floer's classical result \cite{floer89} (see also \cite{fukaya-oh98}), the latter group can be identified with the Morse cohomology of the function $g$, which is the ordinary cohomology of $\bR^{2n}$ relative to any sufficiently negative sublevel set, or equivalently the relative cohomology $H^*(D^n,U_-;\bK) \iso \tilde{H}^{*-1}(U_-;\bK)$.
\end{proof}

Given any integer $\rho \geq 0$, we now apply a sequence of Dehn twists, forming a new Lagrangian sphere
\begin{equation} \label{eq:rho-twist}
\tau_{L^2}^{\rho} \tau_{L^1}(L^0) \htp \tau_{L^2}^{\rho} \tau_{L^0}^{-1}(L^1) \htp \tau_{L^0}^{-1} \tau_{L^2}^{\rho} (L^1).
\end{equation}
Here, the first $\htp$ is the Lagrangian isotopy from \cite[Appendix A]{seidel98b}, which in this case is compatible with orientations because of the assumption on \eqref{eq:hf10}. The second $\htp$ is obvious because $L^0$ and $L^2$ are disjoint (it will be an equality if the supports of the Dehn twists are chosen sufficiently small). We first consider the topological aspect of this:

\begin{lemma} \label{th:homotopy}
Up to isotopy of almost Lagrangian spheres, $\tau_{L^2}^\rho \tau_{L^1}(L^0) $ is independent of $\rho$.
\end{lemma}

\proof Lemma \ref{th:disjoin} shows that one can make $L^2$ disjoint from $L^1$ by an isotopy of embedded spheres. Denote the result by $\tilde{L}^2$, equipping it with the almost Lagrangian structure inherited from the isotopy. Then, $\tau_{\tilde{L}^2}$ is isotopic to $\tau_{L^2}$ through almost symplectomorphisms. But $\tau_{\tilde{L}^2}^\rho(L^1) \htp L^1$ for all $\rho$, hence the result follows by looking at the rightmost expression in \eqref{eq:rho-twist}. \qed

Returning to actual symplectic geometry, \eqref{eq:rho-twist} implies that
\begin{align}
 \HF^*(L^2,\tau_{L^2}^\rho \tau_{L^1}(L^0)) & \iso \HF^{*-\rho(n-1)}(L^2,L^1), \label{eq:back1} \\
\HF^*(W^2,\tau_{L^2}^\rho \tau_{L^1}(L^0)) & \iso \HF^*(\tau_{L^2}^{-\rho}(W^2),L^1).
\label{eq:twisted-floer-group}
\end{align}
The shift in \eqref{eq:back1} appears because $\tau_{L^2}$ reverses the orientation of $L^2$ if $n$ is even. We will need to compute the last-mentioned group more precisely:

\begin{lemma} \label{th:dim-grow}
The total dimension of $\HF^*(W^2,\tau_{L^2}^\rho \tau_{L^1}(L^0))$ is $\rho$ times that of $\tilde{H}^*(U_-;\bK)$.
\end{lemma}

\proof
Let us temporarily restrict to the simplest imaginable example, namely when $M^0 = D^*S^n$ is the unit cotangent bundle of the sphere, $L^0$ is the zero-section, and $K^0$ a fibre. Because the attachment process involves only $n$-handles, $M^2$ is necessarily homotopy equivalent to a wedge of three $n$-spheres. For $n>2$, this directly implies that $c_1(M^2) = 0$. For $n = 2$, we reach the same conclusion after observing that $H_2(M^2) \iso \bZ^3$ is generated by the classes of the Lagrangian spheres $L^k$ ($k = 0,1,2$). Given this, we can equip all our Lagrangian submanifolds with gradings, and then refine the $\bZ/2$-gradings of their Floer cohomology groups to $\bZ$-gradings. More precisely, $L^1$ should be graded so that \eqref{eq:hf10} is nontrivial in degree one; $G$ carries the preferred grading as a graph, and we extend that to $L^2$; finally, $W^2$ is graded so that \eqref{eq:oned} is nontrivial in degree zero. Then Lemma \ref{th:classical-u} and Equation \eqref{eq:back1} hold as isomorphisms of graded groups, the latter thanks to the self-shift formula $\tau_{L^2}(L^2) = L^2[1-n]$ \cite[Lemma 5.7]{seidel99}. From \cite{seidel01} we have a long exact sequence
\begin{multline} \label{eq:LES}
\cdots \rightarrow \HF^*(L^2,L^1) \otimes \HF^*(\tau_{L^2}^{-\rho}(W^2),L^2) \\ \longrightarrow \HF^*(\tau_{L^2}^{-\rho}(W^2),L^1) \longrightarrow \HF^*(\tau_{L^2}^{-(\rho+1)}(W^2),L^1) \rightarrow \cdots
\end{multline}
The left-hand term can be computed explicitly:
\begin{equation} \label{eq:left-term}
\begin{aligned}
\HF^*(L^2,L^1) \otimes \HF^*(\tau_{L^2}^{-\rho}(W^2),L^2) & \iso \HF^*(L^2,L^1) \otimes \HF^*(W^2, \tau_{L^2}^{\rho}(L^2))   \\
& \iso \HF^*(L^2,L^1) \otimes \HF^{*-\rho(n-1)}(W^2,L^2) \\ & \iso \tilde{H}^{*-\rho(n-1)-1}(U_-;\bK).
\end{aligned}
\end{equation}
Since $U_-$ is a compact $(n-1)$-manifold with no closed components, its reduced cohomology is concentrated in degrees $[0,n-2]$, so \eqref{eq:left-term} is concentrated in degrees $[\rho(n-1)+1,(\rho+1)(n-1)]$. Note that these intervals do not overlap. Starting with \eqref{eq:zerod} and applying \eqref{eq:left-term} repeatedly,  we conclude that the left hand term in \eqref{eq:LES} is nonzero only in degrees $> \rho(n-1)$, whereas the middle term is nonzero only in degrees $< \rho(n-1)$. Hence, the map between those two must vanish, which (arguing by induction) implies that
\begin{equation}
\HF^*(W^2,\tau_{L^2}^{\rho} \tau_{L^1}(L^0)) \iso \bigoplus_{i=0}^{\rho-1} \tilde{H}^{*-i(n-1)}(U_-;\bK).
\end{equation}
A priori this is just valid for one example, but any $M^0$ with our general assumptions contains a small cotangent disc bundle around $L^0$, and the resulting $M^2$ then contains the manifold considered in our computation as a Liouville subdomain. By standard convexity arguments, such as \cite[Lemma 7.2]{abouzaid-seidel07}, the Floer cohomology computation can be done inside the smaller space, which completes the argument. \qed

We will also need to carry out this process for several Lagrangian spheres simultaneously. Namely, suppose that $M^0$ contains Lagrangian spheres $(L^0_1,\dots,L^0_m)$, and in addition has the following property:

\begin{assumption} \label{th:co-disc}
Inside $M^0$, there are Lagrangian discs with Legendrian boundary $(W^0_1,\dots,W^0_m)$, such that each $W^0_i$ intersects $L^0_i$ transversely and in a single point (note that $W^0_i$ is allowed to intersect the other $L^0_j$, $j \neq i$, arbitrarily). Moreover, we assume that the boundaries of the $W^0_i$ are mutually disjoint (which is unproblematic, since it can always be achieved by a generic perturbation).
\end{assumption}

Let $M^2$ be the Liouville manifold obtained by carrying out a double handle attachment along each of the $W^0_i$, for $1 \leq i \leq m$ (using the same decomposition $U_\pm$ and function $g$ each time, for simplicity). This contains Lagrangian spheres $L^1_i$ and $L^2_i$, as well as discs $W^2_i$. The analogues of Lemma \ref{th:homotopy} and of the Floer cohomology computations hold for each $i$. In the latter case, this is again by a convexity argument, which shows that when computing Floer cohomology relevant to the $i$-th collection, the presence of the other handle attachments can be ignored.

\subsection{Homologous recombination}
Let $M$ be the fibre of an exact Lefschetz fibration, with vanishing cycles $(V_1,\dots,V_m)$, and total space (after rounding off corners) $E$. Set $M = M^0$ and $L_i^0 = V_i$; assume in addition that $M$ contains discs $W_i^0 = W_i$, such that Assumption \ref{th:co-disc} is satisfied. The resulting $\tilde{M} = M^2$ will be equipped with a new set of vanishing cycles $(\tilde{V}_1,\dots,\tilde{V}_{3m})$, namely
\begin{equation} \label{eq:tilde-collection}
\left\{
\begin{aligned}
& \tilde{V}_{3i-2} = \tau_{L^2_i}^\rho \tau_{L^1_i}(L^0_i), \\
& \tilde{V}_{3i-1} = L^2_i, \\
& \tilde{V}_{3i} = L^1_i
\end{aligned}
\right.
\end{equation}
where $\rho > 0$ is fixed (in principle, one could choose different values of $\rho$ for different $i$, but we will not make use of this additional freedom). Using these, we construct a new Lefschetz fibration, whose total space will be denoted by $\tilde{E}$, with Lefschetz thimbles $(\tilde{\Delta}_1,\dots,\tilde{\Delta}_{3m})$.

\begin{lemma} \label{th:rho-is-1}
If $\rho = 1$, $\tilde{E}$ is deformation equivalent to $E$ (as a Liouville domain).
\end{lemma}

\begin{proof}
Through Hurwitz moves (see for instance \cite{auroux-stable}), which do not affect the deformation class of the total space, we can change our collection \eqref{eq:tilde-collection} as follows:
\begin{equation} \label{eq:hurwitz}
\begin{aligned}
& (\tau_{L^2_1}\tau_{L^1_1}(L^0_1),L^2_1,L^1_1,\dots) \\
& \sim (L^2_1,\tau_{L^1_1}(L^0_1),L^1_1,\dots) \\
& \sim (L^2_1,L^1_1,L^0_1,\dots,L^2_m,L^1_m,L^0_m).
\end{aligned}
\end{equation}
Inside $\tilde{M} = M^2$ we have Lagrangian discs $W^2_i$, each of which intersects $L^2_i$ transversely in a single point, and is disjoint from all the other $L^j_k$. In the local model \eqref{eq:graph}, we have the cotangent fibre $D^n \times \{0\}$, which intersects $G$ transversely in a single point. These fibres give rise to $m$ other Lagrangian discs in $\tilde{M}$, each of which intersects one $L^1_i$ transversely in a single point, and is disjoint from the other $L^j_k$. By applying Lemma \ref{th:cancellation} repeatedly, one reduces the situation to the original collection of vanishing cycles in $M$.
\end{proof}

\begin{lemma}
For any $\rho$, $\tilde{E}$ is almost symplectomorphic to $E$.
\end{lemma}

\begin{proof}
The almost symplectomorphism type of $\tilde{E}$ depends only on the isotopy classes of vanishing cycles as almost Lagrangian submanifolds. In view of Lemma \ref{th:homotopy}, we can therefore assume that $\rho = 1$, and then Lemma \ref{th:rho-is-1} completes the argument.
\end{proof}

As a final elementary remark, note that the complexity of the new manifold is bounded in a way which is independent of $\rho$, $U_\pm$ and $g$:
\begin{equation} \label{eq:5m-complexity}
\complexity(\tilde{E}) \leq \complexity(M) + 5m.
\end{equation}
This follows directly from \eqref{eq:complexity-one} and \eqref{cor:complexity-fibration}, since $\tilde{M}$ is obtained by attaching $2m$ Weinstein handles to $M$, and $\tilde{E}$ is the total space of a Lefschetz fibration with fibre $\tilde{M}$ and $3m$ vanishing cycles.

\begin{lemma} \label{th:all-rho}
Suppose that $\tilde{H}^*(U_-;\bK) = 0$. Then $\SH^*(\tilde{E})$ vanishes if and only if  $\SH^*(E)$ does.
\end{lemma}

\begin{proof}
The assumption implies that $\HF^*(L^2_i,L^1_i) = 0$, hence by using the long exact sequence from \cite{seidel01} also that $\HF^*(L^2_i,\tau_{L^1_i}(L^0_i)) = 0$. In other words, while the two Lagrangian submanifolds $(L^2_i,\tau_{L^1_i}(L^0_i))$ are not geometrically disjoint, they behave as if they were disjoint in so far as the Fukaya category $\Fuk(\tilde{M})$ with coefficients in $\bK$ is concerned. By applying the algebraic expression for Dehn twists \cite[Corollary 17.17]{seidel04}, it follows that up to quasi-isomorphism in that category, $\tau_{L^2_i}^\rho\tau_{L^1_i}(L^0_i)$ is independent of $\rho$. In view of Property \ref{th:meta}, this shows that as far as the computation of $\HW^*(\tilde{\Delta}_i,\tilde{\Delta}_i)$ is concerned, we may just as well assume that $\rho = 1$, in which case $\tilde{E}$ is deformation equivalent to $E$ by Lemma \ref{th:rho-is-1}. Property \ref{th:open-closed} then does the rest.
\end{proof}

\begin{lemma} \label{th:rho-2}
Suppose that $\tilde{H}^*(U_-;\bK) \neq 0$. Then $\SH^*(\tilde{E})$ vanishes if $\rho \geq 2$.
\end{lemma}

\begin{proof}
In view of Property \ref{th:open-closed}, all we need to show is that $\HW^*(\tilde{\Delta}_i, \tilde{\Delta}_i) = 0$ for all $i$. Start with $i = 1$. The fibre $\tilde{M}$ contains the Lagrangian disc $\tilde{W} = W^2_1$, which is disjoint from $\tilde{V}_i$ for all $i>2$ and satisfies
\begin{equation}
\begin{aligned}
& \dim_\bK \HF^*(\tilde{W},\tilde{V}_1) = \rho \, \dim_\bK \tilde{H}^*(U_-;\bK) \\ & >
\dim_\bK \HF^*(\tilde{W},\tilde{V}_2) \cdot \dim_\bK \HF^*(\tilde{V}_2,\tilde{V}_1) =
\dim_\bK \tilde{H}^*(U_-;\bK) \cdot 1.
\end{aligned}
\end{equation}
Hence, by Property \ref{th:hull}, $\HW^*(\tilde{\Delta}_1,\tilde{\Delta}_1) = 0$. In view of Property \ref{th:iterate}, we can do our remaining computations in the total space of the Lefschetz fibration with fibre $\tilde{M}$ and vanishing cycles $(\tilde{V}_2,\dots,\tilde{V}_{3m})$. But the same Lagrangian disc $\tilde{W}$ intersects the first of these cycles in a point, and is disjoint from all the others. Therefore, the wrapped Floer cohomology of the corresponding Lefschetz thimble is zero. The same holds for the next vanishing cycle, by using another Lagrangian disc as in Lemma \ref{th:rho-is-1}. Having now removed the first three vanishing cycles, we note that none of the remaining vanishing cycle intersect  the double handle attached to $W^0_1$.  Since the complement of the handle is a Liouville subdomain, standard convexity arguments and Property \ref{th:meta} imply that all computations can now be performed in the space obtained by attaching $m-1$ double handles to $M$.

At this point, we have reduced the statement to the corresponding one where the original fibration has fibre $M$ and vanishing cycles $(V_2,\dots,V_m)$. If $m = 1$, there are no vanishing cycles left, so the total space of the last-mentioned fibration is $D^2 \times M$ with the corners rounded off, hence its symplectic cohomology is zero by \cite{oancea04}. Otherwise, we argue by induction on $m$.
\end{proof}

\begin{theorem} \label{th:vanish-0}
Let $E$ be the total space of a Lefschetz fibration with fibre $M$ and vanishing cycles $(V_1,\dots,V_m)$. Assume that $\dim(E) = 2n+2 \geq 6$, and that Assumption \ref{th:co-disc} holds. Then there is another Liouville domain $\tilde{E}$, which is almost symplectomorphic to $E$, and such that $\SH^*(\tilde{E})$ vanishes for any coefficient field $\bK$.
\end{theorem}

\begin{proof}
For $n>2$ this is a straightforward consequence of the previous discussion: we can choose $U_- \subset S^{n-1}$ to be homotopy equivalent to the disjoint union of a point and a circle. This satisfies \eqref{eq:euler} but obviously has $\tilde{H}^*(U_-;\bK) \neq 0$ for any $\bK$. Hence, constructing $\tilde{E}$ as before, with $\rho = 2$ for instance, yields the desired result by Lemma \ref{th:rho-2}.

The remaining case $n = 2$ requires only a small variation. Namely, take $U_- \subset S^1$ to consist of two disjoint intervals. One can then choose the extension \eqref{eq:extension} to be Morse and have a single hyperbolic critical point. In fact, the double handle construction then just consists of attaching a Weinstein handle to the boundary of the given disc, and then another one to the co-core sphere of the first handle. Lemma \ref{th:rho-2} still applies, but since \eqref{eq:euler} is violated, Lemma \ref{th:disjoin} and all the topological arguments built on it fail. However, in this particular dimension we know that squares of Dehn twists are isotopic to the identity as almost symplectomorphisms. Hence, changing $\rho$ by an even number does not affect the isomorphism type of $\tilde{E}$ as an almost symplectic manifold. Taking $\rho = 3$ then yields the desired result.
\end{proof}

To refine our result, we want to make use of a particular choice of $U_\pm$. Assume that $n \geq 5$, and fix a natural number $q$. Consider the Moore space obtained by attaching a 2-cell to the circle along a degree $q$ map $S^1 \rightarrow S^1$. We claim that this can be embedded into $S^{n-1}$. For $n > 5$ this is clear, since any generic map is an embedding. For $n = 5$, take explicitly the map from the unit disc $D^2 \subset \bC$ to $\bC^2$ given by $z \mapsto ((1-|z|^2)z,z^q)$. This is an immersion, an embedding in the interior, and a $q$-fold cover along the boundary. Hence, the image is a copy of this Moore space embedded in $\bC^2 \subset S^4$. Given that, take $U_-$ to be the boundary of a regular neighbourhood of the image of the Moore space. Clearly, this satisfies \eqref{eq:euler} and
\begin{equation}
\tilde{H}^*(U_-;\bK) = \begin{cases} 0 & \text{if $\mathrm{char}(\bK)$ does not divide $q$,} \\
\tilde{H}^*(S^1 \vee S^2;\bK) & \text{if $\mathrm{char}(\bK)$ divides $q$.}
\end{cases}
\end{equation}
Therefore, an application of Lemmas \ref{th:all-rho} and \ref{th:rho-2} yields the following:

\begin{theorem} \label{th:selective-vanishing}
Let $E$ be the total space of a Lefschetz fibration with fibre $M$ and vanishing cycles $(V_1,\dots,V_m)$. Assume that $\dim(E) \geq 12$, and that Assumption \ref{th:co-disc} holds. Fix an integer $q \geq 1$. Then there is another Liouville domain $\tilde{E}$ which is almost symplectomorphic to $E$, such that the following holds. If $\bK$ is a field whose characteristic divides $q$, the symplectic cohomology of $\tilde{E}$ with coefficients in $\bK$ vanishes. On the other hand, if the characteristic of $\bK$ does not divide $q$ (this includes characteristic $0$), the symplectic cohomology of $\tilde{E}$ with coefficients in $\bK$ vanishes if and only if the same holds for the original manifold $E$. \qed
\end{theorem}

\begin{example}
Continuing our elementary discussion in Example \ref{th:affine}, the hypotheses of Theorems \ref{th:vanish-0} and \ref{th:selective-vanishing} hold for the important case of affine varieties. We will only give a brief summary, referring to \cite[Sections 16 and 19]{seidel04} and \cite{fukaya-seidel-smith07} for details. Let $X \subset \bC^N$ be a smooth complex affine algebraic variety. Then, a generic linear map $\pi: X \rightarrow \bC$ is a Lefschetz fibration. Denote by $X'$ a smooth fibre of that fibration, which is a hyperplane section of $X$. We can then in turn take another generic linear map $\pi': X' \rightarrow \bC$, which is a Lefschetz fibration. Let $(V_1,\dots,V_m)$ be vanishing cycles for $\pi$. Each such cycle can then be represented as a matching cycle for $\pi'$, associated to a matching path $c_i$ in the plane which joins two critical values of that map. For each $1 \leq i \leq m$, choose a properly embedded half-infinite path $d_i$, whose endpoint is also one of the endpoints of $c_i$, and which otherwise does not intersect $d_i$ or any critical value of $\pi'$ (see Figure \ref{fig:paths} for an illustration). Let $W_i \subset X'$ be the Lefschetz thimble associated to $d_i$. Possibly after a preliminary isotopy, one can arrange that $W_i$ intersects $V_i$ transversely and in a single point. After truncating $X'$ to a compact Liouville domain, one finds that Assumption \ref{th:co-disc} is satisfied.
\end{example}
\begin{figure}
\begin{center}
\epsfig{file=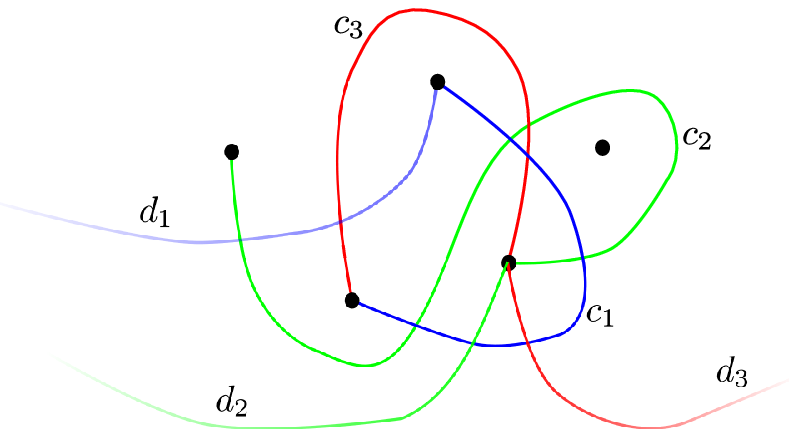}
\caption{\label{fig:paths}}
\end{center}
\end{figure}

With this in mind, Theorem \ref{th:vanish-0} immediately implies Theorem \ref{th:1}, and Theorem \ref{th:selective-vanishing} similarly implies Theorem \ref{th:2}.

\begin{remark}
It is a natural question whether the construction from Theorem \ref{th:selective-vanishing} can be thought of as a localization process away from the primes that divide $q$, or ``tensoring with $\bZ[1/q]$'', in analogy with the well-known notion in homotopy theory \cite{sullivan}. As a possible step towards an answer, take any coefficient field $\bK$ whose characteristic does not divide $q$. Then, the $A_\infty$-structure on $\scrB$ is independent of $\rho$. Hence, by appealing to the full strength of the results about symplectic cohomology from \cite[Appendix]{bourgeois-ekholm-eliashberg09}, one concludes that
\begin{equation} \label{eq:same}
\SH^*(\tilde{E}) \iso \SH^*(E).
\end{equation}
Ongoing work of Bourgeois-Ekholm-Eliashberg aims to understand the product structure on symplectic cohomology from a similar point of view, and one can hope to use that to show that \eqref{eq:same} is an isomorphism of rings. The next question would then be whether the wrapped Fukaya categories of $\tilde{E}$ and $E$, with coefficients in the same field $\bK$, are equivalent at least in some derived sense. A positive answer to this would seem to require an $A_\infty$-refinement of the formula for the wrapped Floer cohomology of Lefschetz thimbles, as well as a generation statement (the methods from \cite{gen} may be helpful in the latter step). As a point of caution, note that one cannot hope for a meaningful relation between the more classical Fukaya categories, whose objects are exact closed Lagrangian submanifolds (since vanishing of symplectic cohomology with $\bZ/2$ coefficients already implies that there cannot be any such submanifolds).
\end{remark}

\section{Nonstandard symplectic structures}

\subsection{Symplectic mapping tori\label{subsec:mclean}}
Having found a method for making symplectic cohomology vanish, we will now balance this by reviewing McLean's construction \cite{mclean09} of nonstandard Liouville structures on flat space which have nonvanishing symplectic cohomology. Along the way we modify some of the steps, replacing the algebro-geometric notion of Kaliman modification by the similar but purely symplectic one of Weinstein two-handle attachment, and removing the reliance on the detailed Conley-Zehnder index computations from \cite{ustilovsky99}.

Let $M$ be a Liouville domain of dimension $2n$. Let $\phi: M \rightarrow M$ be a diffeomorphism which is the identity near $\partial M$, and such that $\phi^*\theta_M - \theta_M$ is the derivative of a function $k$ vanishing near the boundary. Without essential loss of generality (since one can always rescale the symplectic form) we may assume that $|k(x)| < 1$ everywhere. The symplectic mapping torus of $\phi$ is the manifold with corners
\begin{equation} \label{eq:mapping-torus}
[-1,1] \times \bR \times M / \sim, \text{ where } (s,t,x) \sim (s,t-1,\phi(x)).
\end{equation}
We equip this with the one-form
\begin{equation} \label{eq:theta-cone}
\theta_M + s\, dt + d(t\, k) = \theta_M + (s + k) dt + t \, dk,
\end{equation}
which is invariant under the equivalence relation in \eqref{eq:mapping-torus}, and whose exterior derivative is $\omega_M + ds \wedge dt$. By definition, $M$ is a symplectic fibration over $[-1,1] \times S^1$ with a flat symplectic connection, whose holonomy around the circle is $\phi$. The Liouville vector field dual to \eqref{eq:theta-cone} is
\begin{equation} \label{eq:z-torus}
Z_M + (s + k) \partial_s - t X_k,
\end{equation}
where $X_k$ is the Hamiltonian vector field of $k$ (our convention is that $\omega_M(\cdot,X_k) = dX_k$).
The vector field \eqref{eq:z-torus} points strictly outwards along both boundary faces $\{x \in \partial M\}$ and $\{s = \pm 1\}$, the latter thanks to the assumed bound on $k$.

We now want to round off the corners, in order to obtain an actual Liouville domain. Recall that by integrating the Liouville vector field, one gets a collar neighbourhood $N \subset M$ which comes with a preferred diffeomorphism $N \rightarrow [1/2,1] \times \partial M$. Denote the first component of that map by $r: N \rightarrow [1/2,1]$. By construction, $\theta_M|N$ agrees with the pullback of $r(\theta_M|\partial M)$ by that diffeomorphism. Again without essential loss of generality (since one can always deform $\phi$ by conjugating it with the Liouville flow), we may assume that $\phi = \mathrm{id}$ and $k = 0$ on $N$. Define
\begin{equation}
T = \{(s,t,x) \in [-1,1] \times \bR \times M \;:\; \text{either $x \notin N$ or $r(x) \leq \chi(s)$}\}/ \sim,
\end{equation}
where the equivalence relation is as before, and $\chi: [-1,1] \rightarrow [1/2,1]$ is a function with the following properties:
\begin{itemize} \parskip1em
\item $\chi(\pm 1) = 1/2$, $\chi(0) = 1$;
\item $\chi$ is continuous everywhere, and smooth on $(-1,1)$;
\item $\chi' > 0$ on $(-1,0)$, $\chi' < 0$ on $(0,1)$, and $\chi''(0) < 0$;
\item Restricting $\chi$ to $[-1,0]$, all derivatives of its inverse $\chi^{-1}: [1/2,1] \rightarrow [-1,0]$ vanish at the point $1/2$; and the same for the other interval $[0,1]$.
\end{itemize}
Equip $T$ with the restriction of the exact symplectic structure introduced above, denoting the result by $\theta_T$, $\omega_T$ and $Z_T$. The last-mentioned property of $\chi$ ensures that $\partial T$ is smooth. It is easy to see that $Z_T$ points outwards along $\partial T$. In fact, that boundary consists of two parts. One is where $s = \pm 1$ and $x \in \overline{M \setminus N}$. At such boundary points, the characteristic foliation is spanned by the parallel transport vector field $\partial_t$, and the Reeb vector field is
\begin{equation} \label{eq:reeb-1}
R_{\partial T} = \frac{\partial_t}{s + k}.
\end{equation}
Note that this points in positive $\partial_t$ direction if $s = +1$, and in negative $\partial_t$ direction if $s = -1$. The other part of the boundary is where $r(x) = \chi(s)$. At such boundary points, the characteristic foliation is spanned by $R_{\partial M} - \chi'(s) \partial_t$, where $R_{\partial M}$ is the Reeb vector field for $M$; and
\begin{equation} \label{eq:reeb-2}
R_{\partial T} = \frac{R_{\partial M} - \chi'(s) \partial_t}{\chi(s) - s\chi'(s)}.
\end{equation}
Again, if we use the projection $\pi: T \rightarrow [-1,1] \times S^1$ to the $(s,t)$ variables, then the image of \eqref{eq:reeb-2} points in positive $\partial_t$ direction if $s>0$, in negative direction if $s<0$, and is zero if $s = 0$. In the last-mentioned case, $R_{\partial T} = R_{\partial M}$. We conclude:

\begin{lemma} \label{th:reeb-reeb}
Closed Reeb orbits on $\partial T$ whose homology class lies in the kernel of $\pi_*: H_1(T) \rightarrow H_1([-1,1] \times S^1) = \bZ$ correspond bijectively to pairs consisting of a point $t \in S^1$ and a closed Reeb orbit on $\partial M$. \qed
\end{lemma}

We will now introduce a number of additional assumptions.
\begin{itemize} \parskip1em
\item $c_1(M) = 0$ and $H^1(M) = 0$. This allows us to assign Conley-Zehnder indices to closed Reeb orbits in $M$, thereby equipping $\SH^*(M)$ with a canonical integer grading. The same then carries over to $T$;
\item All closed Reeb orbits on $\partial M$ are nondegenerate in the Morse-Bott sense. Moreover,  the space of closed Reeb orbits only has finitely many connected components with a given Conley-Zehnder index;
\item Finally, $M$ contains a closed exact Lagrangian submanifold, which is oriented and {\it Spin}.
\end{itemize}
Using these, we will prove:

\begin{lemma}
The summand $\SH^*(T)_0 \subset \SH^*(T)$ which corresponds to free loops lying in the kernel of $\pi_*: H_1(T) \rightarrow \bZ$ is nonzero, and finite-dimensional in each degree.
\end{lemma}

\begin{proof}
Let $C_j$, $j \geq 1$, be the connected components of the space of parametrized closed Reeb orbits on $\partial M$, ordered so that the length is nondecreasing. Each $C_j$ has a Conley-Zehnder index $\mu_j$, and also carries a local system $\xi_j$ with fibre $\bK$ and holonomy $\pm 1$ (this is important for sign reasons, compare the issue of ``bad orbits'' in contact homology). Inspection of \eqref{eq:reeb-2} shows that the corresponding component of the space of closed Reeb orbits on $\partial T$, which is diffeomorphic to $S^1 \times C_j$ by Lemma \ref{th:reeb-reeb}, is again Morse-Bott nondegenerate, and carries the same Conley-Zehnder index and local system. By standard Morse-Bott methods, there is a spectral sequence converging to $\SH^*(T)_0$ whose starting page is
\begin{equation}
E_1^{pq} = \begin{cases}
(H^*(S^1;\bK) \otimes  H^*(C_{-p};\xi_{-p}))^{p+q-\mu_{-p}} & p < 0, \\
H^q(T;\bK) & p = 0, \\
0 & p > 0.
\end{cases}
\end{equation}
By assumption, the sequence $\mu_j$ takes on the same value only finitely many times. This immediately implies that $\SH^*(T)_0$ is finite-dimensional in each degree.

It remains to show that $\SH^*(T)_0$ is nonzero, or equivalently (since that summand contains the identity element for the ring structure) that $\SH^*(T)$ is nonzero. By assumption, there is a closed exact Lagrangian submanifold $L \subset M$. Again without essential loss of generality, we may assume that $\theta_M|L = 0$, and that $L$ is disjoint from the collar neighbourhood $N$ (the second condition can be satisfied by pushing $L$ inwards via the Liouville flow, and the first condition by adding an exact one-form to $\theta_M$). In that case, the product $W = [-1,1] \times \{0\} \times L$ is an exact Lagrangian submanifold of $T$ with Legendrian boundary. In view of \eqref{eq:reeb-1}, the class of any nontrivial Reeb chord for $\partial W$ has nonzero image under the map $\pi_*: H_1(T,W) \rightarrow H_1([-1;1] \times S^1, [-1,1] \times \{0\}) = \bZ$. By definition of wrapped Floer cohomology, this implies that the wrapped Floer cohomology of $W$ contains its ordinary cohomology as a direct summand, hence is nonzero. By arguing as in the easy direction of Property \ref{th:open-closed}, one sees that  $\SH^*(T)$ must therefore be nonzero.
\end{proof}

By \cite[Theorem 2.5]{etnyre05} (for $n = 1$) or \cite[Theorem 12.4.1]{eliashberg-mishachev} (for $n>1$), there is an isotropic embedded loop in $\partial T$ whose image under $\pi_*$ generates $H_1([-1,1] \times S^1)$. Attach a Weinstein two-handle to that loop, forming a new Liouville domain $U \supset T$. From now on we will assume that $n>1$, so that the handle attachment is subcritical.

\begin{lemma} \label{th:frobenius}
Consider $\SH^*(U)$ with coefficients in a finite field of characteristic $p$. Then, the even degree part is a commutative ring in which the equation $x^p = x$ has a finite number $N \geq 2$ of solutions.
\end{lemma}

\begin{proof}
By the same argument as in \cite[Lemma 7.6]{mclean09}, any solution of $x^p = x$ in the even part of $\SH^*(T)$ is necessarily contained in the subspace $\SH^*(T)_0$, and moreover, in the degree zero part of that subspace. Since that part is finite-dimensional over $\bK$, there are only finitely many such solutions. On the other hand, we know that $\SH^*(T)_0$ is nonzero, so there are at least two solutions (zero and the multiplicative unit). Finally, because our handle attachment is subcritical, the Viterbo restriction map $\SH^*(U) \rightarrow \SH^*(T)$ is an isomorphism \cite{cieliebak02}, so the result carries over to $U$.
\end{proof}

\subsection{Milnor fibres\label{subsec:milnor}}
Again following \cite{mclean09}, we will apply the previous considerations to a class of examples arising from singularity theory. Let $p \in \bC[x_1,\dots,x_{n+1}]$ be a weighted homogeneous polynomial. Weighted homogeneity means that
\begin{equation} \label{eq:weights}
p(\zeta^{w_1} x_1,\dots,\zeta^{w_{n+1}} x_{n+1}) = \zeta^w p(x_1,\dots,x_{n+1})
\end{equation}
for some positive integers $(w_1,\dots,w_{n+1},w)$. Suppose that $p$ has an isolated critical point at the origin (and hence no other critical points). We make two additional assumptions:
\begin{itemize} \parskip1em
\item $w_1 + \cdots + w_{n+1} \neq w$;
\item the link of the singularity is an integral homology sphere.
\end{itemize}
Concrete examples are $p(x) = x_1^2 + \cdots + x_n^2 + x_{n+1}^3$ if $n$ is odd, and $p(x) = x_1^2 + \cdots + x_{n-1}^2 + x_n^3 + x_{n+1}^5$ if $n$ is even \cite{brieskorn66b} (compare with \cite{mclean09}, which requires the link to be a homotopy sphere, hence excludes the case $n = 2$). Consider the Milnor fibre of such a singularity, which by definition is
\begin{equation}
M = \{p(x) = \delta, \; \|x\| \leq 1\}
\end{equation}
where $\delta$ is some nonzero small complex number, with the Liouville structure given by the plurisubharmonic function $\|x\|^2/4$. Because it comes from a hypersurface in $\bC^{n+1}$, $M$ has vanishing first Chern class. As the Milnor fibre of an isolated hypersurface singularity, it is also homotopy equivalent to a wedge of $n$-spheres \cite{milnor68}, hence in particular $H^1(M) = 0$.  Moreover, $M$ always contains a Lagrangian sphere \cite[Lemma 4.14]{seidel99}. By an application of Gray's theorem, the boundary $\partial M$ is isomorphic as a contact manifold to the link of the singularity. We find it convenient to define this link as $p^{-1}(0) \cap h^{-1}(1)$, where $h(x) = w_1|x_1|^2/2 + \cdots + w_{m+1} |x_{m+1}|^2/2$, since then the Reeb flow on the link is precisely the circle action on the left side of \eqref{eq:weights} (the same contact one-form appears in \cite{ustilovsky99}). In particular, the periodic Reeb orbits appear in Morse-Bott nondegenerate components $C_j$, where for some $l>0$, each $C_{j+l}$ differs from $C_j$ by going once around the circle action. The Conley-Zehnder indices $\mu_j$ and $\mu_{j+l}$ then differ by $2(w - w_1 - \cdots - w_{n+1})$, see \cite[Lemma 4.15]{seidel99}. By assumption, this implies that there can only be finitely many connected component of any given index.
Note that while the Reeb flow for the given contact one-form on $\partial M$ may not be the same, one can adjust that by adding a finite conical piece to the boundary. We assume from now on that this has been done, without changing the notation.

The monodromy of the singularity is a symplectic automorphism $\phi: M \rightarrow M$ of the kind considered before. If we take that map and construct the associated Liouville domain $T$ via the symplectic mapping torus, its enlargement to a Liouville manifold is deformation equivalent to the affine variety
\begin{equation} \label{eq:complement}
\bC^{n+1} \setminus p^{-1}(0) \iso \{p(x_1,\dots,x_{n+1})x_{n+2} = 1 \} \subset \bC^{n+2}.
\end{equation}
We will mainly consider the topological implications of this. Because of weighted homogeneity, $\bC^{n+1} \setminus p^{-1}(0)$ is homotopy equivalent to $S^{2n+1} \setminus p^{-1}(0)$, hence by assumption
\begin{equation*}
H_*(T) \iso H^{2n+1-*}(S^{2n+1}, S^{2n+1} \cap p^{-1}(0)) \iso H^{2n+1-*}(S^{2n+1},S^{2n-1}) \iso H_*(S^1).
\end{equation*}
Moreover, the Milnor fibre is simply-connected, which implies that $\pi_1(T)$ is abelian. Once we attach a two-handle to kill $H_1(T)$, the resulting $U$ is therefore contractible. Because we have started with a Stein domain and attached a two-handle, there is a Morse function $h: U \rightarrow (-\infty,0]$ with $h^{-1}(0) = \partial U$ all of whose critical points have index $\leq n+1$. Standing this Morse function on its head and using the fact that $n \geq 2$, one sees that $\partial U$ can be made contractible by attaching cells of dimension $3$ and higher. Hence, $\partial U$ itself must be simply-connected, which by the h-cobordism theorem implies that $U$ is diffeomorphic to $D^{2n+2}$ (this copies a well-known argument about the topology of contractible affine algebraic varieties, see \cite[Theorem 1.3]{zaidenberg98} and references therein). After applying Lemma \ref{th:frobenius}, one arrives at the following conclusion, which is our analogue of \cite[Theorem 7.7]{mclean09}:

\begin{corollary} \label{th:mclean}
$U$ is a Liouville domain diffeomorphic to $D^{2n+2}$. If $\bK$ is finite of characteristic $p$, there is a finite number $N \geq 2$ of solutions of the equation $x^p = x$ in the even part of $\SH^*(U)$. \qed
\end{corollary}

\subsection{Finite type constructions}
We can now prove the remaining results from Section \ref{subsec:results}. The proof of Corollary \ref{th:1.5} uses the product structure on symplectic cohomology, as in \cite{mclean09}, together with homologous recombination. In contrast, that of Corollary \ref{th:2.5} involves only the additive structure, but its use of homologous recombination is more subtle.

\begin{proof}[Proof of Corollary \ref{th:1.5}]
Starting with a given affine algebraic variety $X$, and its associated Liouville domain $E$, we first apply Theorem \ref{th:vanish-0} to construct an almost symplectomorphic Liouville domain $\tilde{E}$ whose symplectic cohomology vanishes. For any $k \geq 0$, take the boundary connect sum of $\tilde{E}$ and $k$ copies of the domain $U$ from Corollary \ref{th:mclean}, and enlarge that to a Liouville manifold $\tilde{X}_k$. This is still almost symplectomorphic to $X$, and by \cite{cieliebak02} it satisfies
\begin{equation}
\SH^*(\tilde{X}_k) \iso \bigoplus_{i=1}^k \SH^*(U).
\end{equation}
This isomorphism is a Viterbo restriction map, hence compatible with the commutative ring structures. Take the coefficient field to be $\bK = \bF_p$ for some $p$. Then, if the even part of $\SH^*(U)$ contains $N \geq 2$ solutions of $x^p = x$, the corresponding number for $\SH^*(\tilde{X}_k)$ is $N^k$. It follows that the $\tilde{X}_k$ are pairwise non-symplectomorphic.
\end{proof}

\begin{remark}
If one is willing to settle for a single nonstandard structure instead of an infinite sequence, the argument can be simplified considerably. Namely, if the symplectic cohomology of the given affine variety vanishes, take its boundary connect sum with $U$; otherwise, kill the symplectic cohomology using Theorem \ref{th:vanish-0}. Either way, one gets a manifold which is not symplectomorphic to the original one.
\end{remark}

\begin{proof}[Proof of Corollary \ref{th:2.5}]
Start with $T$ as in Section \ref{subsec:milnor}. We know that $\SH^*(T) \neq 0$ with coefficients in an arbitrary $\bK$. Moreover, the enlargement of $T$ to a Liouville manifold yields an affine variety \eqref{eq:complement}, so Picard-Lefschetz theory applies. Fix a prime number $q$. By using Theorem \ref{th:selective-vanishing}, we can find a $\tilde{T}_q$ which is almost symplectomorphic (and in particular diffeomorphic) to $T$, and such that $\SH^*(\tilde{T}_q)$ with coefficients in $\bK = \bF_p$ vanishes if and only if $p = q$. Because all the $\tilde{T}_q$ are obtained by homologous recombination applied a fixed Lefschetz fibration, their complexity is bounded, see \eqref{eq:5m-complexity}. Attach a two-handle to $\partial\tilde{T}_q$ to form another Liouville domain $\tilde{U}_q$ with the same symplectic cohomology, and which is diffeomorphic to the disc.  Now, given some affine variety, produce $\tilde{E}$ as in the proof of Corollary \ref{th:1.5}, take the boundary connect sum of that with $\tilde{U}_q$, and enlarge the result to a Liouville manifold $\tilde{X}_q$. Since $\SH^*(\tilde{X}_q) \iso \SH^*(\tilde{T}_q)$, these manifolds are pairwise non-symplectomorphic. Moreover, in view of \eqref{eq:complexity-one}, the complexity of $\tilde{X}_q$ is bounded above by a number independent of $q$.
\end{proof}

\subsection{Infinite type constructions}
The remaining constructions employ methods similar to Corollary \ref{th:2.5}, together with the idea from \cite{mclean06}.

\begin{proof}[Proof of Theorem \ref{th:3}]
Slightly generalizing the argument from Corollary \ref{th:2.5}, we can find, for each integer $q > 1$, a Liouville domain $\tilde{U}_q$ which is diffeomorphic to a disc of our given dimension, and whose symplectic cohomology with coefficients in $\bF_p$ vanishes if and only if $p$ divides $q$. Write $\bP = \{p_1,p_2,\dots\}$ and $q_k = p_1\cdots p_k$. For any $k \geq 1$, let $W_k$ be the boundary connect sum of one copy of $\tilde{U}_{q_1}, \tilde{U}_{q_2}, \dots, \tilde{U}_{q_k}$ each. By attaching a finite conical piece to the boundary if necessary, one can arrange that $W_k$ is contained in the interior of $W_{k+1}$ for all $k$. Let $W$ be the union of all the $W_k$, which is a Liouville manifold. By definition \eqref{eq:inverse-limit} and \cite{cieliebak02},
\begin{equation} \label{eq:infinite-product}
\SH^*(W) \iso \prod_k \SH^*(W_k).
\end{equation}
If we take coefficients in $\bF_p$ where $p \in \bP$, then almost all of the terms in the product \eqref{eq:infinite-product} are zero. Each term is the symplectic cohomology of a Liouville domain, hence of at most countable dimension, and the same will hold for $\SH^*(W)$. On the other hand, if we take coefficients in $\bF_p$ where $p \notin \bP$, then all terms in \eqref{eq:infinite-product} are nonzero. But an infinite product of nonzero vector spaces is always of uncountable dimension.
\end{proof}

Corollary \ref{th:3.5} is proved in exactly the same way, except that when forming $W_k$, one additionally takes the boundary connect sum with the Liouville domain obtained by truncating the given finite type Liouville manifold $M$.


\begin{thebibliography}{10}

\bibitem{gen}
M.~Abouzaid.
\newblock A geometric criterion for generating the Fukaya category.
\newblock Preprint arXiv:1001.4593, 2010.

\bibitem{abouzaid-seidel07}
M.~Abouzaid and P.~Seidel.
\newblock An open string analogue of {V}iterbo functoriality.
\newblock {\em Geom. Topol.}, 14:627--718, 2010.

\bibitem{auroux-stable}
D.~Auroux.
\newblock A stable classification of Lefschetz fibrations.
\newblock {\em Geom. Topol.}, 9:203--217, 2005.

\bibitem{biran-cieliebak01}
P.~Biran and K.~Cieliebak.
\newblock Lagrangian embeddings into subcritical {S}tein manifolds.
\newblock {\em Israel J. Math.}, 127:221--244, 2002.

\bibitem{sft-compactness}
F.~Bourgeois, Ya.~Eliashberg, H.~Hofer and K.~Wysocki.
\newblock Compactness results in symplectic field theory.
\newblock {\em Geom. Topol.}, 7:799--888, 2003.

\bibitem{bourgeois-ekholm-eliashberg09}
F.~Bourgeois, T.~Ekholm, and Ya.~Eliashberg.
\newblock Effect of {L}egendrian surgery.
\newblock Preprint arXiv:0911.0026, 2009.

\bibitem{bourgeois-oancea09}
F.~Bourgeois and A.~Oancea.
\newblock An exact sequence for contact- and symplectic homology.
\newblock {\em Invent. Math.}, 175:611--680, 2009.

\bibitem{brieskorn66b}
E.~Brieskorn.
\newblock Beispiele zur {D}ifferentialtopologie von {S}ingularit\"aten.
\newblock {\em Invent. Math.}, 2:1--14, 1966.

\bibitem{capecchi89}
M.~Capecchi.
\newblock Altering the genome by homologous recombination.
\newblock {\em Science}, pages 1288--1292, 1989.

\bibitem{cerf68}
J.~Cerf.
\newblock {\em Sur les diff{\'e}omorphismes de la sph{\`e}re de dimension trois
  {$(\Gamma_4 = 0)$}}, volume~53 of {\em Lecture Notes in Math.}
\newblock Springer, 1968.

\bibitem{cerf70}
J.~Cerf.
\newblock La stratification naturelle des espaces des fonctions
  diff\'erentiables r\'eelles et le th\'eor\`eme de la pseudo-isotopie.
\newblock {\em Publ. Math. IHES}, 39:5--173, 1970.

\bibitem{cieliebak02}
K.~Cieliebak.
\newblock Handle attaching in symplectic homology and the chord conjecture.
\newblock {\em J. Eur. Math. Soc.}, 4(2):115--142, 2002.

\bibitem{cieliebak-eliashberg}
K.~Cieliebak and Ya.~Eliashberg.
\newblock Symplectic geometry of {S}tein manifolds.
\newblock Book manuscript.

\bibitem{cieliebak-floer-hofer95}
K.~Cieliebak, A.~Floer, and H.~Hofer.
\newblock Symplectic homology {II}: a general construction.
\newblock {\em Math. Z.}, 218:103--122, 1995.

\bibitem{cieliebak-frauenfelder-oancea09}
K.~Cieliebak, U.~Frauenfelder, and A.~Oancea.
\newblock Rabinowitz {F}loer homology and symplectic homology.
\newblock Preprint arXiv:0903.0768, 2009.

\bibitem{drinfeld02}
V.~Drinfeld.
\newblock D{G} quotients of {DG} categories.
\newblock {\em J. Algebra}, 272:643--691, 2004.

\bibitem{ekholm08}
T.~Ekholm.
\newblock Rational symplectic field theory over {$\mathbb{Z}_2$} for exact
  {L}agrangian cobordisms.
\newblock {\em J. Eur. Math. Soc.}, 10:641--704, 2008.

\bibitem{eliashberg90}
Ya.~Eliashberg.
\newblock Topological characterization of {S}tein manifolds of dimension
  {$>2$}.
\newblock {\em Internat. J. Math.}, 1:29--46, 1990.

\bibitem{eliashberg94}
Ya.~Eliashberg.
\newblock Symplectic geometry of plurisubharmonic functions.
\newblock In {\em Gauge theory and symplectic geometry}, volume 488 of {\em
  NATO Adv. Sci. Inst. Ser. C Math. Phys. Sci.}, pages 49--67. Kluwer Acad.
  Publ., 1997.

\bibitem{eliashberg-gromov93}
Ya.~Eliashberg and M.~Gromov.
\newblock Convex symplectic manifolds.
\newblock In {\em Several complex variables}, volume~52 of {\em Proc. Symposia
  Pure Math.}, pages 135--162. Amer. Math. Soc., 1991.

\bibitem{eliashberg-mishachev}
Ya.~Eliashberg and N.~Mishachev.
\newblock {\em Introduction to the {$h$}-principle}.
\newblock Amer. Math. Soc., 2002.

\bibitem{sft-intro}
Ya.~Eliashberg, A.~Givental and H.~Hofer.
\newblock Introduction to Symplectic Field Theory.
\newblock {\em Geom. Func. Anal.}, Special Volume, Part II, 560--673, 2000.

\bibitem{etnyre05}
J.~Etnyre.
\newblock Legendrian and transversal knots.
\newblock In {\em Handbook of knot theory}, pages 105--185. Elsevier, 2005.

\bibitem{floer89}
A.~Floer.
\newblock Witten's complex and infinite dimensional {M}orse theory.
\newblock {\em J. Differential Geom.}, 30:207--221, 1989.

\bibitem{fukaya-oh98}
K.~Fukaya and Y.-G. Oh.
\newblock Zero-loop open strings in the cotangent bundle and {M}orse homotopy.
\newblock {\em Asian J. Math.}, 1:96--180, 1998.

\bibitem{fukaya-seidel-smith07}
K.~Fukaya, P.~Seidel, and I.~Smith.
\newblock Exact {L}agrangian submanifolds in simply-connected cotangent
  bundles.
\newblock 172:1--27, 2008.

\bibitem{fukaya-seidel-smith07b}
K.~Fukaya, P.~Seidel, and I.~Smith.
\newblock The symplectic geometry of cotangent bundles from a categorical
  viewpoint.
\newblock In {\em Homological Mirror Symmetry: New Developments and
  Perspectives}, volume 757 of {\em Lecture Notes in Physics}, pages 1--26.
  Springer, 2008.

\bibitem{giroux02}
E.~Giroux.
\newblock G\'eom\'etrie de contact: de la dimension trois vers les dimensions
  sup\'erieures.
\newblock In {\em Proceedings of the International Congress of Mathematicians
  (Beijing), Vol. II}, pages 405--414. Higher Ed. Press, 2002.

\bibitem{poly}
H.~Hofer, K.~Wysocki and E.~Zehnder.
\newblock A general Fredholm theory. I. A splicing-based differential geometry.
\newblock {\em J. Eur. Math. Soc.}, 9:841--876, 2007.

\bibitem{johns09}
J.~Johns.
\newblock Morse-bott handle attachments and plumbing.
\newblock 2009.
\newblock Preprint.

\bibitem{keller99a}
B.~Keller.
\newblock On the cyclic homology of exact categories.
\newblock {\em J. Pure Appl. Algebra}, 136:1--56, 1999.

%
\bibitem{ks}
M.~Kontsevich and Y.~Soibelman.
\newblock Notes on $A_\infty$-algebras, $A_\infty$-categories and non-commu\-ta\-tive geometry.
\newblock In {\em Homological mirror symmetry}, Lect. Notes in Physics vol. 757, pages 153--219. Springer, 2009.

\bibitem{lefevre}
K.~Lefevre.
\newblock {\em Sur les {$A_\infty$}-cat\'egories}.
\newblock PhD thesis, Universit\'e Paris 7, 2002.

\bibitem{lyubashenko-ovsienko06}
V.~Lyubashenko and S.~Ovsienko.
\newblock A construction of quotient {$A_\infty$}-categories.
\newblock {\em Homol. Homot. Appl.}, 8:157--203, 2006.

\bibitem{maydanskiy09}
M.~Maydanskiy.
\newblock {\em Exotic symplectic manifolds from {L}efschetz fibrations}.
\newblock PhD thesis, MIT, 2009.

\bibitem{maydanskiy-seidel09}
M.~Maydanskiy and P.~Seidel.
\newblock {L}efschetz fibrations and exotic symplectic structures on cotangent
  bundles of spheres.
\newblock {\em J. Topology}, 3:157--180, 2010.

\bibitem{mclean06}
M.~McLean.
\newblock Second year dissertation, Cambridge University, 2006.

\bibitem{mclean08}
M.~McLean.
\newblock Private communications, 2008-10.

\bibitem{mclean09}
M.~McLean.
\newblock Lefschetz fibrations and symplectic homology.
\newblock {\em Geom. Topol.}, 13:1877--1944, 2009.

\bibitem{milnor63}
J.~Milnor.
\newblock {\em Morse theory}.
\newblock Princeton Univ. Press, 1963.

\bibitem{milnor68}
J.~Milnor.
\newblock {\em Singular points of complex hypersurfaces}.
\newblock Princeton Univ. Press, 1968.

\bibitem{oancea04b}
A.~Oancea.
\newblock A survey of {F}loer homology for manifolds with contact type boundary
  or symplectic homology.
\newblock {\em Ensaios Mat.}, 7:51--91, 2004.

\bibitem{oancea04}
A.~Oancea.
\newblock The {K}unneth formula in {F}loer homology for manifolds with
  restricted contact type boundary.
\newblock {\em Math. Ann.}, 334:65--89, 2006.

\bibitem{quillen88}
D.~Quillen.
\newblock Algebra cochains and cyclic cohomology.
\newblock {\em Publ. Math. IHES}, 68:139--174, 1988.

\bibitem{ritter10}
A.~Ritter.
\newblock Topological quantum field theory structure on symplectic cohomology.
\newblock Preprint arXiv:1003.1781, 2010.

\bibitem{seidel98b}
P.~Seidel.
\newblock Lagrangian two-spheres can be symplectically knotted.
\newblock {\em J. Differential Geom.}, 52:145--171, 1999.

\bibitem{seidel99}
P.~Seidel.
\newblock Graded {L}agrangian submanifolds.
\newblock {\em Bull. Soc. Math. France}, 128:103--146, 2000.

\bibitem{seidel01}
P.~Seidel.
\newblock A long exact sequence for symplectic {F}loer cohomology.
\newblock {\em Topology}, 42:1003--1063, 2003.

\bibitem{seidel08}
P.~Seidel.
\newblock {$A\sb \infty$}-subalgebras and natural transformations.
\newblock {\em Homology, Homotopy Appl.}, 10:83--114, 2008.

\bibitem{seidel07}
P.~Seidel.
\newblock A biased survey of symplectic cohomology.
\newblock In {\em Current Developments in {M}athematics ({H}arvard, 2006)},
  pages 211--253. Intl.\ Press, 2008.

\bibitem{seidel04}
P.~Seidel.
\newblock {\em {F}ukaya categories and {P}icard-{L}efschetz theory}.
\newblock European Math. Soc., 2008.


\bibitem{seidel06}
P.~Seidel.
\newblock {S}ymplectic homology as {H}ochschild homology.
\newblock In {\em {A}lgebraic {G}eometry: {S}eattle 2005}, volume~1, pages
  415--434. Amer.\ Math.\ Soc., 2008.

\bibitem{seidel-smith04b}
P.~Seidel and I.~Smith.
\newblock The symplectic topology of {R}amanujam's surface.
\newblock {\em Comment. Math. Helv.}, 80:859--881, 2005.

\bibitem{smale59}
S.~Smale.
\newblock Diffeomorphisms of the $2$-sphere.
\newblock {\em Proc. Amer. Math. Soc.}, 10:621--626, 1959.

\bibitem{sullivan}
D.~Sullivan.
\newblock {\em Geometric topology: localization, periodicity and {G}alois symmetry (the 1970 MIT notes).}
\newblock $K$-Monographs in Mathematics, vol. 8, Springer, 2005.

\bibitem{ustilovsky99}
I.~Ustilovsky.
\newblock Infinitely many contact structures on {$S^{4m+1}$}.
\newblock {\em Internat. Math. Res. Notices}, (14):781--791, 1999.

\bibitem{viterbo97b}
C.~Viterbo.
\newblock Functors and computations in {F}loer homology with applications,
  {P}art {II}.
\newblock Preprint, 1996.

\bibitem{viterbo97a}
C.~Viterbo.
\newblock Functors and computations in {F}loer homology with applications,
  {P}art {I}.
\newblock {\em Geom. Funct. Anal.}, 9:985--1033, 1999.

\bibitem{weinstein91}
A.~Weinstein.
\newblock Contact surgery and symplectic handlebodies.
\newblock {\em Hokkaido Math. J.}, 20:241--251, 1991.

\bibitem{zaidenberg98}
M.~Zaidenberg.
\newblock Lectures on exotic algebraic structures on affine spaces.
\newblock Preprint math.AG/\-9801075, 1998.

\end{thebibliography}

\end{document}